\documentclass[a4paper,reqno,11pt]{amsart}

\newfont{\cyr}{wncyr10 scaled 1100}

\usepackage[left=2.7cm,right=2.7cm,top=3.5cm,bottom=3cm]{geometry}

\usepackage{amsthm,amssymb,amsmath,amsfonts,mathrsfs,amscd}
\usepackage{mathtools}
\usepackage{extarrows}
\usepackage[latin1]{inputenc}
\usepackage[all]{xy}
\usepackage{latexsym}
\usepackage{longtable}
\usepackage{xcolor}
\usepackage{comment}
\usepackage[shortlabels]{enumitem}
\usepackage{setspace}
\setdisplayskipstretch{}
\usepackage{multirow}
\usepackage{xcolor,colortbl}
\usepackage{changepage}
\usepackage{float}
\usepackage{tikz-cd}
\usepackage[colorlinks=true,
            linkcolor=black,
            citecolor=black,
            urlcolor=black]{hyperref}
\usepackage[backend=biber,style=alphabetic,maxnames=99,minnames=99, sorting=nyt]{biblatex}

\DeclareFieldFormat[article]{title}{#1}
\DefineBibliographyStrings{english}{
  in = {}
}
\AtBeginBibliography{\footnotesize}

\theoremstyle{plain}
\newtheorem{theorem}{Theorem}[section]
\newtheorem{corollary}[theorem]{Corollary}
\newtheorem{lemma}[theorem]{Lemma}
\newtheorem{proposition}[theorem]{Proposition}
\newtheorem{conjecture}[theorem]{Conjecture}

\theoremstyle{definition}
\newtheorem{definition}[theorem]{Definition}

\newtheorem{examplewr}[theorem]{Example}

\theoremstyle{remark}
\newtheorem{obswr}[theorem]{Observation}
\newtheorem{remarkwr}[theorem]{Remark}

\definecolor{Gray}{gray}{0.85}
\definecolor{LightCyan}{rgb}{0.88,1,1}

\newcolumntype{g}{>{\columncolor{Gray}}c}
\newcolumntype{y}{>{\columncolor{LightCyan}}c}
\newcolumntype{o}{>{\columncolor{pink}}c}

\newenvironment{remark}{\begin{remarkwr}\begin{upshape}}{\end{upshape}\end{remarkwr}}

\newcommand{\cO}{{\mathcal O}}

\newcommand{\sH}{\mathscr H}

\newcommand{\sX}{{\mathscr X}}

\newcommand{\Z}{\mathbb{Z}}
\newcommand{\Q}{\mathbb{Q}}
\newcommand{\R}{\mathbb{R}}
\newcommand{\C}{\mathbb{C}}

\newcommand{\PP}{\mathbb{P}}
\newcommand{\X}{\mathbb{X}}
\newcommand{\D}{\mathbb{D}}

\newcommand{\GL}{\mathrm{GL}}

\newcommand{\SL}{\mathrm{SL}}
\newcommand{\SO}{\mathrm{SO}}
\newcommand{\Norm}{\mathrm{N}}

\newcommand{\Gal}{\mathrm{Gal\,}}

\newcommand{\Hom}{\mathrm{Hom}}

\newcommand{\ord}{\mathrm{ord}}

\newcommand{\coInd}{\mathrm{coInd}}

\newcommand{\fp}{\mathfrak{p}}

\newcommand{\too}{\longrightarrow}								
\newcommand{\mapstoo}{\longmapsto}
\newcommand{\intoo}{\lhook\joinrel\longrightarrow}

\DeclareRobustCommand\ontoo{\relbar\joinrel\twoheadrightarrow}

\makeatletter
\newsavebox{\@brx}
\newcommand{\llangle}[1][]{\savebox{\@brx}{\(\m@th{#1\langle}\)}%
  \mathopen{\copy\@brx\kern-0.5\wd\@brx\usebox{\@brx}}}
\newcommand{\rrangle}[1][]{\savebox{\@brx}{\(\m@th{#1\rangle}\)}%
  \mathclose{\copy\@brx\kern-0.5\wd\@brx\usebox{\@brx}}}
\makeatother

\begin{document}
	
	\include{thebibliography}

\title[Eisenstein class of a torus bundle and log-rigid classes for $\SL_n(\Z)$]{Eisenstein class of a torus bundle and \\ log-rigid analytic classes for $\SL_n(\Z)$}
\author{Mart\'i Roset and Peter Xu}

\address{M.~R.: IMJ-PRG, Sorbonne Universit\'e, Paris, France}
\email{marti.roset-julia@imj-prg.fr}
\address{P.~X.: UCLA, Los Angeles, USA}
\email{peterx@math.ucla.edu}

\subjclass[2020]{11F75, 11G15, 11S40, 11S80}

\begin{abstract}
	Starting from a topological treatment of the Eisenstein class of a torus bundle, we define log-rigid analytic classes for $\SL_n(\Z)$. These are group cohomology classes for $\SL_n(\Z)$ valued on log-rigid analytic functions on Drinfeld's $p$-adic symmetric domain. Such classes can be evaluated at points attached to totally real fields of degree $n$ where $p$ is inert. We conjecture that these values are $p$-adic logarithms of Gross--Stark units in the narrow Hilbert class field of totally real fields. We provide evidence for the conjecture by comparing our constructions to $p$-adic $L$-functions. In addition, we prove it in certain situations where the totally real field is Galois over $\Q$, as a consequence of the fact that in this case there is a conjugate of a Gross--Stark unit in $\Q_p$.
\end{abstract}

\maketitle

\tableofcontents

\section{Introduction}

The values of modular units at CM points, called elliptic units, have rich arithmetic significance. Notably, they generate abelian extensions of imaginary quadratic fields. In \cite{DD}, Darmon and Dasgupta proposed a conjectural construction of elliptic units for real quadratic fields and predicted that they behave similarly to elliptic units. Their construction consists of a $p$-adic limiting process involving periods along real quadratic geodesics of logarithmic derivatives of modular units.  

Using a different approach, Dasgupta used Shintani's method to extend this construction to totally real fields of arbitrary degree in \cite{DasguptaShintani}, and Dasgupta and Kakde proved that this recipe gives $p$-units in abelian extensions of totally real fields \cite{DKBrumerStark}, \cite{DKExplicitCFT}. More precisely, they proved that their resulting objects are Gross--Stark units. Remarkably, their work provides a solution to Hilbert's twelfth problem for totally real fields via $p$-adic methods.

Darmon, Pozzi, and Vonk constructed analogs of modular functions, called rigid classes, which can be evaluated at real quadratic points, and expressed the original construction of \cite{DD} as the value of a rigid class in \cite{DPV2}. Their work accentuates the parallel between Gross--Stark units and elliptic units, as rigid classes play the role of modular functions. Moreover, it leads to an alternative proof of the conjecture of \cite{DD} in the real quadratic setting.

In this paper, we construct a log-rigid analytic class for $\SL_n(\Z)$ and study its values at points attached to totally real fields where $p$ is inert. We conjecture that these values are $p$-adic logarithms of Gross--Stark units and provide evidence for it by comparing our constructions to $p$-adic $L$-functions and using the rank $1$ Gross--Stark conjecture, proven in \cite{DDP} and \cite{Ventullo}.

Moreover, we prove the conjecture in certain situations where the totally real field is Galois over $\Q$, as a consequence of the fact that in this case there is a conjugate of a Gross--Stark unit in $\Q_p$. To our knowledge, this extends the type of abelian extensions of $F$ that can be constructed explicitly using only the values of the derivatives of $p$-adic $L$-functions. Moreover, it expresses Gross--Stark units as values of a modular-like object, namely the log-rigid class.

A key ingredient in our construction is the Eisenstein class of a torus bundle of Bergeron, Charollois, and Garc\'ia \cite{BCG} and its pullbacks by torsion sections, that replace the role of modular units. {In particular, we conjecture that Gross--Stark units can be obtained via a $p$-adic limiting process involving periods of Eisenstein classes on locally symmetric spaces attached $\SL_n(\R)$ along tori determined by totally real fields.} We hope that this represents a first step toward a modular, or more technically automorphic, construction of Gross--Stark units for totally real fields, which would generalize the results of \cite{DD} and \cite{DPV2}.

\subsection{Siegel units and abelian extensions of quadratic fields}\label{subsec: Siegel Units and abelian extensions of imaginary quadratic fields}
We begin by explaining the construction of Siegel units and their relation with the theory of complex multiplication for imaginary quadratic fields. 
Let $E$ be an elliptic curve defined over a scheme $S$, fix a positive integer $c$ coprime to $6$, and denote by $\mathbb{N}^{(c)}$ the set of positive integers coprime to $c$. 
\begin{proposition}\label{prop: construction of Siegel units}
	There exists a unique function ${}_c \theta_E \in \cO(E - E[c])^\times$ satisfying:
	\begin{enumerate}
		\item The divisor of ${}_c \theta_E$ is $E[c] - c^2(0)$.
		\item ${}_c \theta_E$ is invariant under pushforward induced by multiplication by $a$ for all $a \in \mathbb{N}^{(c)}$.
	\end{enumerate}
\end{proposition}

Let $N \geq 3$ be a positive integer coprime to $c$, denote by $\Gamma(N) \subset \SL_2(\Z)$ the congruence subgroup of full level $N$, and let $\sH$ be the complex upper half-plane. We can then consider the universal elliptic curve
\[
E \coloneq \Gamma(N) \backslash \left((\sH \times \C)/\Z^2\right) \too Y(N)\coloneq \Gamma(N) \backslash \sH. 
\]
The proposition above yields the function ${}_c \theta_E \in \cO(E - E[c])^\times$, which can be used to construct modular units on $\Gamma(N) \backslash \sH$ in the following way. A vector $v \in \Q^2/\Z^2 - \{ 0\}$ of order $N$ induces a torsion section $v\colon Y(N) \to E - E[c]$. Then, the pullback ${}_c g_v \coloneq v^\ast({}_c \theta_E) \in \cO(Y(N))^\times$ is called a {\em Siegel unit} and is an instance of a modular unit. It gives rise to a $\Gamma(N)$-invariant function on $\sH$, that we will denote by the same symbol. The theory of complex multiplication implies that the values of Siegel units at special points have deep significance. 
\begin{theorem}\label{thm: CM values of Siegel units}
	Let $\tau \in \sH$ be a CM point attached to a quadratic imaginary field $K$, i.e. $\tau$ is stabilized by a subgroup of norm one elements $K^1 \subset \SL_2(\Q)$ of $K$. Then, 
	\[
	{}_c g_v (\tau) \in K^{\mathrm{ab}} \subset \bar{\Q}.
	\]
\end{theorem}
An important question, included as the twelfth problem on Hilbert's famous list, is to find an analog of this theorem for general number fields. The case of real quadratic fields has been extensively studied by various methods. We are particularly interested in the $p$-adic approach initiated by Darmon and Dasgupta in \cite{DD} and followed, among others, by Darmon, Pozzi, and Vonk in \cite{DPV2}. We next outline these works in a language suited to this paper.

Let $F$ be a real quadratic field and $p$ a rational prime. Observe that $\sH$ does not contain real quadratic points, i.e. there are no points stabilized by a torus of norm one elements $F^1 \subset \SL_2(\Q)$ of $F$. On the other hand, $\sH$ has geodesics stabilized by these tori. Moreover, if $(z, \gamma z) \subset \sH$ is a segment of such geodesic, where $\gamma \in F^1 \cap \Gamma(p^r)$, and  $v \in \Q^2/\Z^2 - \{0\}$ is of exact order $p^r$, we have the so-called Meyer's theorem
\begin{equation}\label{eq: periods of dlog Siegel unit}
	\frac{1}{2\pi i}\int_{z}^{\gamma z} \mathrm{dlog}({}_c g_v) = \zeta_c(F, [\mathfrak b] , 0) \in \Z.
\end{equation}
Here, $\zeta_c(F, [\mathfrak b], 0)$ denotes the value at $s = 0$ of a $c$-smoothed partial zeta function attached to $F$ and an ideal class $[\mathfrak b]$ in a narrow class group of conductor divisible $p^r$, determined by the inclusion $F^1 \subset \SL_2(\Q)$ and $v$. In addition to encoding information about abelian extensions of totally real fields, these zeta values possess notable $p$-adic properties and serve for the construction of measures that determine $p$-adic partial zeta functions of $F$.

The search for a symmetric space containing real quadratic points, combined with the $p$-adic properties of the partial zeta values considered above, leads to replacing $\sH$ by a $p$-adic symmetric space to generalize Theorem \ref{thm: CM values of Siegel units}. More precisely, if we let $\sH_p\coloneq \mathbb{P}^1(\C_p) - \mathbb{P}^1(\Q_p)$ be the $p$-adic upper half-plane and $\mathcal A$ its ring of rigid analytic functions, we have: 
\begin{itemize}
	\item $\sH_p$ contains points stabilized by $F^1 \subset \SL_2(\Q)$ if and only if $p$ is nonsplit in $F$.
	\item There is a $\GL_2(\Q_p)$-equivariant isomorphism between $\mathcal A^\times/\C_p^\times$ and the space of $\Z$-valued measures on $\PP^1(\Q_p)$ of total mass zero (\cite{vdP82}), suggesting that $\mathcal A^\times$ encodes information about $p$-adic zeta functions and refinements of their values.
\end{itemize} 
In \cite{DPV2}, Darmon, Pozzi, and Vonk exploit the distribution relation of Siegel units attached to vectors of arbitrary $p$-power order to construct a cohomology class
\[
\mathcal J_{\mathrm{DR}} \in H^1(\SL_2(\Z), \mathcal A^\times).
\] 
This class can be viewed as a generalization of a modular function. Indeed, the space of invariant functions $H^0(\SL_2(\Z), \mathcal A^\times) = \C_p^\times$ is too simple, suggesting to study the first cohomology group instead. Moreover, if $\tau \in \sH_p$ is stabilized by $F^1 \subset \SL_2(\Q)$ and $F^1 \cap \SL_2(\Z) = \langle \pm \gamma_\tau \rangle$, they define the value $\mathcal J_{\mathrm{DR}}[\tau] \coloneq \mathcal J_{\mathrm{DR}}(\gamma_\tau)(\tau) \in \C_p^\times$. Fix embeddings $\bar{\Q} \subset \bar{\Q}_p \subset \C_p$.

\begin{theorem}[Darmon--Pozzi--Vonk]\label{thm: main thm of DPV2}
	Let $\tau \in \sH_p$ be as above with stabilizer $\langle \pm \gamma_\tau \rangle \subset \SL_2(\Z)$ be attached to a real quadratic field $F$ where $p$ is inert, and suppose that $\{\tau , 1\}$ generate a fractional ideal of $F$. Then,
	\[
	\log_p(\mathcal J_{\mathrm{DR}}[\tau]) = \log_p(u), \ u \in H = \text{narrow Hilbert class field of $F$} \subset F^{\mathrm{ab}} \subset \bar{\Q}. 
	\]
\end{theorem}
This theorem provides a level $1$ version of Theorem \ref{thm: CM values of Siegel units} for real quadratic fields where $p$ is inert. Indeed, it produces nontrivial elements in abelian extensions of real quadratic fields as values of $\mathcal J_{\mathrm{DR}}$ at special points in $\sH_p$.

The class $\mathcal J_{\mathrm{DR}}$ is the unique lift via the quotient map $\mathcal A^\times \to \mathcal A^\times/\C_p^\times$ of the restriction to $\SL_2(\Z)$ of a class $J_{\mathrm{DR}} \in H^{1}(\SL_2(\Z[1/p]), \mathcal A^\times /\C_p^\times)$, also constructed in \cite{DPV2}. Moreover, the Hecke module $H^{1}(\SL_2(\Z[1/p]), \mathcal A^\times /\C_p^\times)_{\Q}$ is isomorphic to the sum of $H^1(\Gamma_0(p), \Q)$ and an Eisenstein line. The lift $\mathcal{J}_{\mathrm{DR}}$ is important to define the values of $J_{\mathrm{DR}}$ and it is the object we aim to generalize in this work. Here and for the rest of the paper, the subindex $\Q$ denotes the tensor product with $\Q$ over $\Z$.

\subsection{Construction of the log-rigid class for $\SL_n(\Z)$}
The work of Bergeron, Charollois, and Garc\'ia in \cite{BCG} provides a generalization of logarithmic derivatives of Siegel units which is relevant for the study of totally real fields of degree $n$: the {\em Eisenstein class of a torus bundle}. Let $E \to X$ be an oriented real vector bundle of rank $n$ over an oriented manifold $X$. Suppose that $E$ contains a sub-bundle $E_\Z$ with fibers isomorphic to $\Z^n$. We can then construct the torus bundle $T \coloneq E/E_\Z \to X$. Consider the following class in singular cohomology with $\Z$-coefficients
\[
T[c] - c^n \{ 0 \} \in H^0(T[c]) \simeq H^n(T , T - T[c]),
\]
where the isomorphism above is the Thom isomorphism. The long exact sequence in relative cohomology provides a map $H^{n-1}(T - T[c]) \to H^n(T, T - T[c])$. The Eisenstein class ${}_c z_T$ attached to $T$ and $c$ is constructed from the next theorem and is analogous to the functions ${}_c \theta_E$ determined in Proposition \ref{prop: construction of Siegel units}.

\begin{theorem}[\cite{BCG}]\label{thm: Eisenstein class of a torus bundle}
	There exists a unique class ${}_c z_T \in H^{n-1}(T - T[c] , \Z[1/c])$ satisfying:
	\begin{enumerate}
		\item ${}_c z_T$ is a lift of $T[c] - c^n \{ 0\} \in H^n(T , T - T[c], \Z[1/c])$.
		\item ${}_c z_T$ is invariant under pushforward induced by multiplication by $a$ for all $a \in \mathbb{N}^{(c)}$.
	\end{enumerate}
\end{theorem}
Let $\sX \coloneq \SL_n(\R)/\SO_n$ be the symmetric space attached to $\SL_n(\R)$, let $v_r \in \Q^n/\Z^n$ be the column vector $(1/p^r, 0 , \dots, 0)^t$ and let $\Gamma_r$ be its stabilizer in $\Gamma \coloneq \SL_n(\Z)$. Finally, fix $q$ an auxiliary integer such that the full level congruence subgroup $\Gamma(q)$ is torsion-free and $[\Gamma: \Gamma(q)]$ is prime to $p$, which imposes that $p$ is sufficiently large. Then, $\Gamma_r(q)\coloneq \Gamma_r \cap \Gamma(q)$ is torsion-free. We can apply the previous theorem to the universal family of tori
\[
T_r \coloneq \Gamma_r(q) \backslash (\sX \times \R^n / \Z^n) \too \Gamma_r(q) \backslash \sX
\]
and obtain the Eisenstein class ${}_c z_{T_r}$, that we will simply denote by $z_r$. The vector $v_r$ induces a torsion section $v_r \colon \Gamma_r(q) \backslash \sX \to T_r - T_r[c]$ and we can consider the pullback $v_r^\ast z_r$, which defines a $\Gamma_r$-invariant class on $\Gamma_r(q) \backslash \sX$. This class is a {higher-dimensional analog of} $\mathrm{dlog} {}_c g_v$. 

The pullbacks of Eisenstein classes by $p$-power torsion sections satisfy distribution relations parallel to those of Siegel units. In particular, $(v_r^\ast z_r)_r$ are compatible with respect to pushforward by the projection maps. Using these properties and Shapiro's lemma, we package the pullbacks of the Eisenstein classes by $p$-power torsion sections in a group cohomology class 
\[
\mu_0 \in H^{n-1}(\Gamma, \D_0(\X, \Z[1/m]))^{w = -1},
\]
where $m$ is a multiple of $c$  prime to $p$, $\D_0(\X, \Z[1/m])$ is the space of $\Z[1/m]$-valued measures on $\X \coloneq \Z_p^n - p\Z_p^n$ of total mass zero, and $w$ denotes the involution given by the action of $\GL_n(\Z)/\SL_n(\Z)$. We will sometimes refer to $\mu_0$ as an {\em Eisenstein cocycle}, following precedent in the literature, which we briefly review and compare with our approach in Section~\ref{subsection : litreview}.

The class $\mu_0$ valued in $\D_0(\X, \Z[1/m])$ is suitable for the construction of rigid classes via a Poisson kernel. Let $\sX_p \coloneq \mathbb{P}^{n-1}(\C_p) - \bigcup_{H \in \mathcal H} H$ be Drinfeld's $p$-adic symmetric domain, where $\mathcal H$ is the set of all $\Q_p$-rational hyperplanes. Denote by $\mathcal A_{\mathcal L}$ the space of log-rigid analytic functions. Informally, $\mathcal A_{\mathcal L}$ consists of the $\C_p$-valued functions on $\sX_p$ such that its restriction to any affinoid is of the form 
\[
(\text{rigid analytic function}) + \sum_{H, H' \in \mathcal \mathcal H} c_{H, H'}\log_p\left(\ell_H(z) / \ell_{H'}(z) \right),
\] 
where $c_{H, H'} \in \Q_p$ are all but finitely many $0$, $\ell_H(z)$ denotes the equation of the hyperplane $H \in \mathcal H$, and $\log_p\colon \C_p^\times \to \C_p$ is the $p$-adic logarithm satisfying $\log_p(p) = 0$. Integration over $\X$ leads to a $\Gamma$-equivariant lift
\begin{equation}\label{eq: ST lift}
	\mathrm{ST}\colon \D_0(\X, \Z_p) \too \mathcal A_{\mathcal L}, \ \lambda \mapstoo \left( z \mapstoo \int_\X \log_p( z^t \cdot x  ) d\lambda \right)
\end{equation} 
This lift is an instance of the theory of $p$-adic Poisson kernels, which relate measures on the set of hyperplanes of $\Q_p^n$, or on $\Z_p^\times$-bundles over them, to functions on $\sX_p$. We refer the reader to the works of Schneider--Teitelbaum \cite{SchneiderTeitelbaum}, van der Put \cite{vdP82}, and Gekeler \cite{Gekeler2020} for more examples of this phenomenon.
We define our desired log-rigid analytic class as
\[
J_{E, \mathcal L} \coloneq \mathrm{ST}(\mu_0) \in H^{n-1}(\Gamma, \mathcal A_{\mathcal L}).
\]
The construction of \( J_{E, \mathcal L} \) can be compared to that in \cite{DPV2} when $n = 2$, leading to \( J_{E, \mathcal L} = \log_p(\mathcal J_{\mathrm{DR}}) \). In particular, this shows that the class \( \log_p(\mathcal J_{\mathrm{DR}}) \) can be constructed solely from logarithmic derivatives of Siegel units, rather than from the full Siegel units.

\subsection{Values of $J_{E, \mathcal L}$ at totally real fields where $p$ is inert} Recall the embeddings $\bar{\Q} \subset \bar{\Q}_p \subset \C_p$. Let $F$ be a totally real field of degree $n$ where $p$ is inert, and let $\tau \in F^n$ be such that its coordinates give an oriented $\Z$-basis $\mathfrak a^{-1}$, for $\mathfrak a$ an ideal of $\cO_F$. Since $p$ is inert, it follows that $\tau \in \sX_p$. Moreover, $\tau$ is a special point in $\sX_p$ in the sense that its stabilizer in $\SL_n(\Q)$ is isomorphic to the norm $1$ elements of $F$. In particular, its stabilizer in $\Gamma$ is an abelian group of rank $n-1$. Following a similar recipe than the case $n = 2$, we define the evaluation of $J \in H^{n-1}(\Gamma, \mathcal A_{\mathcal L})$ at $\tau \in \sX_p$, giving $J[\tau] \in \C_p$. From our construction, one readily deduces $J_{E, \mathcal L}[\tau] \in F_p$ and the theorem below gives evidence regarding the arithmetic significance of this value. 

Let $H$ be the narrow Hilbert class field of $F$. The embedding $\bar{\Q} \subset \bar{\Q}_p$ determines a prime $\mathfrak p$ of $H$ above $p\cO_F$. Denote by $\cO_H[1/p]_{-}^\times$ the subgroup of $p$-units of $H$ where every complex conjugation of $H$ acts by $-1$. Attached to $\mathfrak p$ and $c$, there is a Gross--Stark unit $u \in \cO_H[1/p]^\times_{-} \otimes \Q$, whose valuations at primes above $p$ are related to $c$-smoothed partial $L$-functions of the extension $H/F$. In fact, the proof of the Brumer--Stark conjecture in \cite{DKBrumerStark} and \cite{DKSW} ensures that $u \in \cO_H[1/p]^\times_-$ under certain minor assumptions on $c$, that we will assume for the rest of the introduction (see Remark \ref{rmk: from Gross--Stark to Brumer--Stark}).

\begin{theorem}\label{thm: Tr J[tau] = Tr log u}
	For $n \geq 2$, $\mathrm{Tr}_{F_p/\Q_p}J_{E, \mathcal L}[\tau] = \mathrm{Tr}_{F_p/\Q_p}\log_p(u^{\sigma_{\mathfrak a}})$, where $u \in \cO_H[1/p]_{-}^\times$ is the Gross--Stark unit given above and $\sigma_{\mathfrak a} \in \Gal(H/F)$ is the Frobenius corresponding to $\mathfrak a$.
\end{theorem}
The proof of this result uses that the integral of $v_r^\ast z_r$ along the $(n-1)$-dimensional submanifold of $\Gamma_r(q) \backslash \sX$ determined by the inclusion $F^1 \subset \SL_n(\Q)$ is a special value of a partial zeta function of $F$, generalizing \eqref{eq: periods of dlog Siegel unit}. From there, we construct the $p$-adic partial zeta function of $F$ attached to $\mathfrak a$ from $\mu_0$ and express $\mathrm{Tr}_{F_p/\Q_p}J_{E, \mathcal L}[\tau]$ as its derivative at $s = 0$. Thus, Theorem \ref{thm: Tr J[tau] = Tr log u} follows from the Gross--Stark conjecture in rank $1$, proved in \cite{DDP} and \cite{Ventullo}.

The previous theorem, together with Theorem \ref{thm: main thm of DPV2} involving real quadratic fields, suggests:
\begin{conjecture}\label{conj: JEis[tau] = log(u_tau)}
	We have $
	J_{E, \mathcal L}[\tau] = \log_p(u^{\sigma_{\mathfrak a}})$, where $u\in  \cO_H[1/p]^\times_{-}$ and $\sigma_{\mathfrak a} \in \Gal(H/F)$ are as above.
\end{conjecture} 
We outline some evidence towards the conjecture. When $F/\Q$ is Galois, the Gross--Stark units satisfy $\log_p(\sigma_\mathfrak a u) \in \Q_p$ if the narrow ideal class $[\mathfrak a]$ is $\Gal(F/\Q)$-stable. This is studied in detail in \cite{RX2.5}, but we include a discussion in Section \ref{subsec: The case of Galois extensions} for completeness. Moreover, it can be deduced from the properties of $J_{E, \mathcal L}$ that $J_{E, \mathcal L}[\tau] \in \Q_p$ if $\tau$ generates an ideal $\mathfrak a^{-1}$ that is $\Gal(F/\Q)$-stable. Thus, it follows from Theorem \ref{thm: Tr J[tau] = Tr log u} that the values of $J_{E, \mathcal L}$ can be used to calculate the $p$-adic logarithm of the Gross--Stark unit $u$ in this setting.

\begin{theorem}\label{thm intro: JEis[tau]=log_p(u^sigma a) in the Gal case}
    Suppose $F/\Q$ is Galois, $p$ is inert in $F$, and $\tau$ generates an ideal $\mathfrak a$ that is $\Gal(F/\Q)$-stable. Then, $J_{E, \mathcal L}[\tau] = \log_p(u^{\sigma_\mathfrak a}) \in \Q_p$ for the Gross--Stark unit $u \in \cO_H[1/p]^\times_{-}$ introduced above.
\end{theorem}
In particular, Theorem \ref{thm intro: JEis[tau]=log_p(u^sigma a) in the Gal case} applies to real quadratic fields, recovering instances of \cite[Theorem B]{DPV2} when the coordinates of $\tau$ generate a $\Gal(F/\Q)$-fixed ideal of $\cO_F$. Note that this result implies that we can obtain a formula for certain Gross--Stark units in the narrow Hilbert class field of $F$ only from derivatives of $p$-adic $L$-functions in settings where $F_p$ is not equal to $\Q_p$, see Remark \ref{rmk: from log(u) to u and narrow Hilbert class field Galois case}. Moreover, the relevant units involved in the constructions appear as values of the modular-like object $J_{E, \mathcal L}$, supporting the parallel between $J_{E, \mathcal L}$ and Siegel units. In ongoing work, we are exploring which ramified abelian extensions of $F$ can be constructed using a generalization of this observation.

{The proof of Conjecture \ref{conj: JEis[tau] = log(u_tau)} would give a construction of Gross--Stark units using Eisenstein classes defined purely from the topology of torus bundles. Moreover, it is possible to find explicit representatives of the classes considered in this article via an integral symbol complex (similarly to \cite{X1}), which we will present in a sequel to this article. Ultimately, such formulas can be related to those obtained via Shintani's method, as we mention below, and from there to the formulas for Gross--Stark units of \cite{DasguptaShintani} and \cite{DasguptaSpiess2018} proven in \cite{DKExplicitCFT} and \cite{DHS}. We hope that this paves the way for proving Conjecture \ref{conj: JEis[tau] = log(u_tau)}.  

The symbol complex methods mentioned above also seem to shed light on the construction of rigid analytic classes for $\SL_n(\Z[1/p])$ lifting our log-rigid class, which we will present in the sequel. This would fully generalize the construction of $J_{\mathrm{DR}}$ of \cite{DPV2} to $\SL_n$, and suggest a different approach to Conjecture \ref{conj: JEis[tau] = log(u_tau)}, namely to generalize the strategy of \cite{DPV2}. In a different direction, Darmon, Gehrmann, and Lipnowski generalized the theory of rigid \emph{meromorphic} classes of \cite{DarmonVonk2021} to the setting of orthogonal groups in \cite{DGL}, see \cite{GGM} for an extensive list of values of these classes at special points. 

\begin{remark}
    In this paper, we focused on invariants that conjecturally belong to the narrow Hilbert class field of $F$. It would be interesting to explore the following extensions of the construction. First, we could evaluate $J_{E, \mathcal L}$ at points $\tau \in \sX_p$ represented by $\tau \in F^n$ whose coordinates give a $\Q$-basis of $F$, but they do not necessarily span an ideal of $F$ over $\Z$. In this case, we expect to construct invariants in ramified abelian extensions of $F$. For example, when $n = 2$, such invariants are conjectured to belong to ring class fields of $F$. Second, when $n$ is odd, Gross--Stark units in the narrow Hilbert class field of $F$ are trivial. It seems that to construct meaningful invariants also in the setting when $n$ is odd, we would need a higher-level version of $J_{E, \mathcal L}$, generalizing \cite{chapdelaine2009punits} to totally real fields. For such construction, we expect that the corresponding invariants belong to ray class fields of $F$. 
\end{remark}

\subsection{Related cocycles in the literature}\label{subsection : litreview}
There are constructions of similar cohomology classes to $\mu_0$ in the literature, frequently under the name of {\em Eisenstein cocycles}. Notably, the work of Sczech \cite{Sczech1993} together with its integral refinement by Charollois and Dasgupta \cite[Theorem 4]{CharolloisDasgupta2014}, and the works of Charollois, Dasgupta, Greenberg, and Spiess (see \cite{CharolloisDasguptaGreenberg2015} and \cite{DasguptaSpiess2018}) using Shintani's method give explicit formulas for Eisenstein cocycles. These works yield cocycles for $S$-arithmetic groups; on the other hand, they take values in measures on $\X$ together with some additional data, such as a set of linear forms in $n$-variables (used for $Q$-summation), or the set of rays in $\R^n$ not generated by a vector in $\Q^n$.

More directly related to our approach is the work of Beilinson, Kings, and Levin \cite{beilinson2018} in the equivariant cohomology of a torus, its adelic refinement by Galanakis and Spiess \cite{GalanakisSpiess2024}, as well as the results of Bannai et. al. in equivariant Deligne cohomology \cite{Bannai_2024}. These articles also define equivariant Eisenstein classes by specifying residues in a torus bundle but work with larger and more general coefficient modules, such as the logarithm sheaf or a variant of it. In this way, the first two articles construct distribution-valued cohomology classes by delicate topological considerations. Our Eisenstein class is closely related to the specialization to trivial coefficients of these classes, see Remark \ref{rmk: relation to BKL}. Some computations in cohomology will afford us a lift of our class with finer properties, e.g. a total-mass zero condition, making it sufficient to construct log-rigid classes and produce a conjectural formula for Gross--Stark units.  

The latter article \cite{Bannai_2024} works equivariantly under a nonsplit torus associated to a particular totally real field (rather than a general linear group), and relates a de Rham regulator of this class to $L$-values closely tied to a method of Shintani. This suggests that $\mu_0$, or its restriction to a nonsplit torus, can be compared to the cocycles of \cite{CharolloisDasguptaGreenberg2015}, \cite{DasguptaSpiess2018}, \cite{Spiess}.

\subsection{Structure of the paper}
In Section \ref{section:singular}, we define the Eisenstein class of a torus bundle and prove a distribution relation involving the pullbacks of this class by torsion sections. In Section \ref{sec: differential form representatives of Eisenstein class}, we introduce an explicit differential form representing the Eisenstein class for a universal family of tori following \cite{BCG}. We use it to prove that the sum of the pullbacks of this form along the torsion sections of exact order $p$ is $0$. The content of these two sections is combined in Section \ref{sec: Topological construction of the Eisenstein group cohomology class} to construct the class $\mu_0 \in H^{n-1}(\Gamma, \D_0(\X, \Z[1/m]))^{-}$. In Section \ref{sec: Drinfeld's $p$-adic symmetric domain and rigid cocycles}, we construct the log-rigid class $J_{E, \mathcal L}$ from $\mu_0$ and define its values at points attached to totally real fields where $p$ is inert. In Section \ref{sec: Values of the log-rigid class and Gross--Stark conjecture}, we prove Theorem \ref{thm: Tr J[tau] = Tr log u}, relating the local trace of these values, to the derivative of a $p$-adic $L$-function, and therefore to local traces of $p$-adic logarithms of Gross--Stark units. Finally, in Section \ref{sec: conjecture on values of JEis}, we state Conjecture \ref{conj: JEis[tau] = log(u_tau)}, study the case where $F/\Q$ is Galois, and include a numerical example for real quadratic fields that illustrates the behaviour of Gross--Stark units. 

\subsection{Acknowledgments} We would like to thank Henri Darmon for suggesting that we work on this project. We are also indebted to him and Pierre Charollois for their constant advice during its development. We are further very grateful to Nicolas Bergeron, Romain Branchereau, Luis Garc\'ia, Lennart Gehrmann, Eyal Goren, Hazem Hassan, Pierre Morain, Peter Patzt, Brent Pym, and Jan Vonk for stimulating discussions and key suggestions on the subject. We would also like to thank Mateo Crabit Nicolau for his advice when performing numerical computations with Gross--Stark units, and H\r{a}vard Damm-Johnsen, Max Fleischer, and Yijia Liu for publicly sharing their algorithms that compute Gross--Stark units for real quadratic fields. Marti Roset received support from an FRQNT Doctoral Training Scholarship.

\section{Eisenstein class of a torus bundle}\label{section:singular}

In this section, we introduce the Eisenstein class of a torus bundle, as studied in \cite{BCG}. We focus specifically on the torus bundle
\[
\Gamma' \backslash (\sX \times \R^n/\Z^n ) \too \Gamma' \backslash \sX,
\]
where $\sX$ is the symmetric space attached to $\SL_n(\R)$ and $\Gamma' \subset \Gamma\coloneq \SL_n(\Z)$ is a congruence subgroup that is torsion-free. We then prove several properties of this class, including a distribution relation between its pullbacks by torsion sections, which parallels the distribution relations satisfied by Siegel units. Unless stated otherwise, in this section, we consider singular cohomology with $\Z$-coefficients.

\subsection{Thom and Eisenstein classes of a torus bundle}
Let $\pi\colon E \to X$ be an oriented real vector bundle of rank $n$ over an oriented manifold $X$. Since $E$ is oriented, for every fiber $E_x \subset E$ over $x \in X$ we have a preferred generator 
\[
u_{E_x} \in H^n(E_x, E_x - \{ 0\}) \simeq \Z
\]
satisfying a local compatibility condition (see \cite[Page 96]{MilnorStasheff}). The Thom isomorphism theorem asserts that there is a global class that restricts to the orientation of each fiber. 

\begin{theorem}[Thom isomorphism theorem]\label{thm: thom iso for vector bundles}
	There is a unique class $u_E \in H^n(E, E - \{ 0\})$ such that its pullback to any fiber $E_x$ of $E$ is equal to $u_{E_x}$. Moreover, for every $i \in \Z$, we have an isomorphism 
	\[
	H^i(X) \xlongrightarrow{\sim} H^{i + n}(E, E - \{ 0\}), \ y \mapstoo \pi^\ast y \smile u_E.
	\]
\end{theorem}
\begin{proof}
    See Section 10, and in particular Theorem 10.4, of \cite{MilnorStasheff}.
\end{proof}

Now suppose that $E$ contains a sub-bundle $E_\Z$ with fibers isomorphic to $\Z^n$. We can then construct the torus bundle $T \coloneq E/E_{\Z} \to X$. For every $x \in X$, the orientation on $E_x$ yields an orientation on $T_x$. Fix $c \in \Z_{\geq 1}$ and consider the class
\[
u_{T_x,c} \in H^n(T_x, T_x - T_x[c]) \simeq \bigoplus_{z \in T_x[c]} H^n(T_x, T_x - \{z\}),
\]
which restricts to the generator of each $H^n(T_x, T_x - \{z\})$ determined by the orientation of $T_x$ at $z$, for every $z \in T_x[c]$. 
By considering a tubular neighborhood of $T[c]$ in $T$, and applying the excision theorem, we deduce the Thom isomorphism for torus bundles from Theorem \ref{thm: thom iso for vector bundles} above. 

\begin{theorem}\label{thm: Thom iso for a torus bundle relative to c torsion}
	There is a unique class $u_{T,c} \in H^n(T, T - T[c])$ such that its pullback to any fiber $T_x$ of $T$ is equal to $u_{T_x, c}$. 
    Moreover, for every $i \in \Z$, the Thom isomorphism in Theorem \ref{thm: thom iso for vector bundles} induces an isomorphism
	\[
	H^i(T[c]) \xlongrightarrow{\sim} H^{i+n}(T, T - T[c]).
	\]
\end{theorem}
\begin{definition}
	The class $u_E$ is called the Thom class of the bundle $E \to X$, and $u_{T,c}$ is the Thom class of the torus bundle $T \to X$ relative to the $c$-torsion.
\end{definition}
We now outline the definition of the Eisenstein class of the torus bundle $T \to X$ relative to the $c$-torsion. For this, we assume that for all $i \in \Z$, the group $H^i(X)$ is finitely generated. Consider the following class in singular cohomology 
\[
T[c] - c^n \{ 0\} \in H^0(T[c]).
\]
Denote by the same symbol the image of this class in $H^n(T, T - T[c])$ via the Thom isomorphism given in Theorem \ref{thm: Thom iso for a torus bundle relative to c torsion}. The long exact sequence in relative cohomology gives 
\begin{equation}\label{eq: les relative cohomology}
	\cdots \too H^{n-1}({T}) \too H^{n-1}({T} - {T}[c]) \too H^n({T} , {T} - {T}[c]) \too H^n({T}) \too \cdots.
\end{equation}
We then have the following theorem.
\begin{theorem}[Sullivan, Bergeron--Charollois--Garc\'ia]\label{thm: Eisenstein class of a torus bundle}
	There exists a unique class ${}_c z_T \in H^{n-1}(T - T[c] , \Z[1/c])$ satisfying:
	\begin{enumerate}
		\item It is a lift of $T[c] - c^n \{ 0\} \in H^n(T , T - T[c], \Z[1/c])$ by the map in \eqref{eq: les relative cohomology}.
		\item It is invariant under pushforward induced by multiplication by $a$ in $T$ for all $a \in \mathbb{N}^{(c)}$.
	\end{enumerate}
\end{theorem}
\begin{proof}
	Section 2 and Section 3 of \cite{BCG} prove the existence of the class ${}_c z_T$ with coefficients in $\Z[1/N]$, for $N$ divisible by $c$ and coprime to $p$ (see the remarks below Lemma 9 and Definition 10 of \cite{BCG}). This is sufficient for our purposes, but we refer the reader to \cite[Page 14]{xu2023arithmetic} for a proof that the coefficients can be taken to be $\Z[1/c]$.
\end{proof}
\begin{definition}
	The class ${}_c z_T$ above is the Eisenstein class attached to $T$ and $c$.
\end{definition}
Throughout this work, we will define invariants attached to totally real fields of degree $n$ from periods of Eisenstein classes of torus bundles of rank $n$. 

\begin{remark}
Theorem \ref{thm: Eisenstein class of a torus bundle} has the following visual interpretation. The first point is equivalent to the fact that the image of $T[c]- c^n \{ 0 \}$ in $H^n(T, \Z[1/c])$ vanishes. Informally, this means that there is a codimension $n-1$ submanifold $\Sigma \subset T - T[c]$ such that 
\[
\partial \Sigma = t(T[c]- c^n \{ 0\}), \ t \in \Z,
\]
where $\partial \Sigma$ denotes the boundary of $\Sigma$. On the other hand, the class of $\Sigma$ is not unique, and the second point of the theorem provides a preferred class, ${}_c z_T$ with this property. In particular, ${}_c z_T$ allows defining linking numbers with $T[c] - c^n \{ 0\}$ as the intersection number with the preferred choice of $\Sigma$.
\end{remark}

\begin{remark}\label{rmk: details on pushforward by [a]}
	Let $a \in \mathbb{N}^{(c)}$. Consider inclusion 
	\[
	i\colon T - T[ac] \xhookrightarrow{\quad} T - T[c].
    \]
	Multiplication by $a$ induces a map
	\[
	[a]\colon T - T[ac] \too T - T[c].
	\]
	The map pushforward induced by multiplication by $[a]$ on $H^i(T - T[c])$ appearing in Theorem \ref{thm: Eisenstein class of a torus bundle} is defined as the composition
	\[
	H^{i}(T - T[c]) \xlongrightarrow{i^\ast} H^{i}(T - T[ac]) \xlongrightarrow{[a]_\ast} H^{i}(T - T[c]).
	\]
	We similarly define $[a]_\ast\colon H^i(T , T-T[c]) \xrightarrow{} H^i(T, T - T[c])$.
\end{remark}
\subsection{Eisenstein class of universal families of tori}\label{subsec: Eisenstein class of universal family of tori with level structure}\label{subsec: topo construction of Eis class}

Let $n \geq 2$ and denote by $\sX \coloneq \SL_n(\R)/ \SO_n$ the symmetric space attached to $\SL_n(\R)$. We are interested in the Eisenstein class of universal families of tori over quotients of $\sX$ by the following congruence subgroups.

Let $p$ be an odd prime such that $(p, c) = 1$, for $r \geq 0$ consider the column vector 
\[
v_r \coloneq (1/p^r, 0 , \dots, 0)^t \in \Q^n/\Z^n,
\]
and let $\Gamma_r$ be its stabilizer in $\Gamma \coloneq \SL_n(\Z)$. Fix $q \neq p$ an auxiliary prime such that the full level congruence subgroup $\Gamma(q) \subset \Gamma$ is torsion-free and has index prime to $p$. Observe that these conditions imply that $p$ is sufficiently large. Finally, define $\Gamma_r(q) \coloneq \Gamma_r \cap \Gamma(q)$ and consider the torus bundle
\[
T_r \coloneq \Gamma_r(q) \backslash (\sX \times \R^n / \Z^n) \too \Gamma_r(q) \backslash \sX.
\]
\begin{definition}
	Denote by $z_r \coloneq {}_c z_{T_r} \in H^{n-1}(T_r - T_r[c] , \Z[1/c])$ the Eisenstein class attached to the torus bundle $T_r$ and $c$.
\end{definition}
\begin{remark}
    We introduced the auxiliary prime $q$ and the congruence subgroups $\Gamma_r(q)$ to ensure that their action on $\sX$ is free, which holds as $\Gamma_r(q)$ is torsion-free. Thus, the fibers of $T_r$ are $n$-tori. 
\end{remark}

\begin{remark} \label{rem:dependc}
    We are omitting $c$ from the notation since it is generally fixed. We note in passing also that the dependence of our classes on $c$ is simple: as explained in \cite[Section 3.3]{BCG}, it follows from the definition of the Eisenstein class that, if $c$ and $d$ are coprime, 
    \[
    ([c]^*-c^n){}_dz_r = ([d]^*-d^n){}_cz_r.
    \]
\end{remark}

For $r \geq 1$, the vector $v_r$ induces a section 
\[
v_r\colon \Gamma_r(q)\backslash \sX \too {T}_r - T_r[c], \ [g] \mapstoo [ (g , v_r)].
\]
We can then consider the pullback $v_r^\ast z_r \in H^{n-1}(\Gamma_r(q) \backslash \sX , \Z[1/c])$. We proceed to study the behavior of $v_r^\ast z_r$ with respect to two different actions. Observe that $\Gamma_r(q)$ is a normal subgroup of $\Gamma_r$. Thus, we can define an action of $\Gamma_r$ on $\Gamma_r(q) \backslash \sX$ as follows. For $\gamma \in \Gamma_r$, 
\[
\gamma\colon \Gamma_r(q) \backslash \sX \too \Gamma_r(q) \backslash \sX, \ [g] \mapstoo [\gamma g].
\]
As a consequence, $\Gamma_r$ acts on $H^{n-1}(\Gamma_r(q) \backslash \sX , \Z[1/c])$ via pullback. Since $\Gamma_r$ fixes $v_r$, we deduce that the class $v_r^\ast z_r$ is fixed by this action, as we make precise in the next lemma.
\begin{lemma}\label{lemma: v_r*z_r is Gamma_r equivariant}
	Consider the same notation as above. We have
	\[
	v_r^{\ast} z_r \in H^{n-1}(\Gamma_r(q) \backslash \sX, \Z[1/c])^{\Gamma_r}.
	\]
\end{lemma}
\begin{proof}
	Let $\gamma \in \Gamma_r$ and define the map of torus bundles
	\[
	\tilde{\gamma}\colon T_r \too T_r, \ [(g, v)] \mapstoo [(\gamma g, \gamma v)].
	\]
    We have $\tilde{\gamma}^\ast T_r[c] = T_r[c]$ and $\tilde{\gamma}^\ast \{0\} = \{0\}$ in $H^0(T_r[c])$. Moreover, since $\tilde{\gamma}^\ast u_{T_r,c} = u_{T_r,c}$, as $\tilde{\gamma}$ is orientation preserving, it follows that 
    \[
    \tilde{\gamma}^\ast(T_r[c] - c^n\{ 0\}) = T_r[c] - c^n \{0\} \in H^n(T_r , T_r - T_r[c]).
    \]
    This implies that $\tilde{\gamma}^\ast z_r$ is a lift of $T_r[c] - c^n \{0\}$. Moreover, for every $a \in \mathbb{N}^{(c)}$, $\tilde{\gamma}^\ast$ commutes with $[a]_\ast$. Indeed, define $\widetilde{\gamma^{-1}}$ in the same way as $\tilde{\gamma}$ but replacing $\gamma$ by $\gamma^{-1}$. Since $\tilde{\gamma}^\ast = \widetilde{\gamma^{-1}}_\ast$, the desired commutativity follows then from taking the pushforward of $[a] \circ \widetilde{\gamma^{-1}} = \widetilde{\gamma^{-1}} \circ [a]$. From there, we deduce that $\tilde{\gamma}^\ast z_r$ is invariant under $[a]_\ast$. Thus, Theorem \ref{thm: Eisenstein class of a torus bundle} implies $z_r = \tilde{\gamma}^\ast z_r$. Pulling back this equality by $v_r\colon \Gamma_r(q) \backslash \sX \to T_r - T_r[c]$ yields the desired expression.
\end{proof}

Let $w = \mathrm{diag}(1, -1, 1, \dots, 1) \in \GL_n(\Z)$. Since $w$ normalizes $\Gamma_r(q)$ and $\SO_n$, conjugation induces the following map 
\[
w\colon \Gamma_r(q) \backslash \sX \too \Gamma_r(q) \backslash \sX, \ [g] \mapstoo [w g w^{-1}],
\]
which induces an involution $w$ on $H^{n-1}(\Gamma_r(q) \backslash \sX, \Z[1/c])$ via pullback. Here and for the rest of the section, we will denote with a superindex $-$ the $w = -1$ eigenspace for $w$.

\begin{lemma}\label{lemma: action of w on pullback of Eis class}
    For every $r \geq 1$, we have
    \[
    v_r^\ast z_r \in H^{n-1}(\Gamma_r(q) \backslash \sX, \Z[1/c])^{-}.
    \]
\end{lemma}
\begin{proof}
    The proof is analogous to the proof of Lemma \ref{lemma: v_r*z_r is Gamma_r equivariant}, so we only outline it. Denote by $\tilde{w}$ the morphism of torus bundles
    \[
    \tilde{w}\colon T_r \too T_r, \ \left[ ([g] , v)   \right] \mapstoo [ (wgw^{-1} , wv)   ].
    \]
    Since $\tilde{w}$ reverses the orientation on the fibers (because the determinant of the matrix defining $w$ is $-1$), it follows that 
    \[
    \tilde{w}^\ast(T_r[c] - c^n \{0\}) = - \left( T_r[c] - c^n \{0\} \right) \in H^n(T_r , T_r - T_r[c]).
    \]
    Similar to Lemma \ref{lemma: v_r*z_r is Gamma_r equivariant}, we deduce from there that $\tilde{w}^\ast z_r = -z_r$. The desired result follows by pulling back this equality by $v_r$ and observing $w v_r = v_r$.    
\end{proof}

\subsection{Distribution relations}
We give some compatibility properties regarding the classes $z_r$ and their pullbacks by torsion sections. In particular, we prove a distribution relation. We begin with the following general lemma. 

\begin{lemma}\label{lemma: general lemma to commute pushforward and pullback}
	Consider the commutative diagram of topological spaces, where all the maps are continuous
	\[
	\begin{tikzcd}
		Z \arrow[d, "{h_1}"] \arrow[r, "{f_1}"] & Y \arrow[d, "{h_2}"] \\
		X \arrow[r, "{f_2}"]                                  & {S}       .                          
	\end{tikzcd}
	\]
	Suppose that the following two conditions hold:
	\begin{enumerate}
		\item $h_1$ and $h_2$ are $r$-sheeted covering maps, for $r \in \Z_{\geq 1}$.
		\item If $x \in X$ and $\{z_j\}_{j = 1}^r$ are its distinct lifts by $h_1$, the images $\{f_1(z_j)\}_{j= 1}^r$ are distinct.
	\end{enumerate}
	Then, for all $i \in \Z_{\geq 0}$, we have
	\[ 
	(h_1)_\ast f_1^\ast = f_2^\ast (h_2)_\ast \colon H^i(Y) \too H^i(X).
	\] 
\end{lemma}
\begin{proof}
	For the proof of this lemma, we follow the same notation as in its statement. Let $\varphi \in C^i(Y, \Z)$ be a degree $i$ cochain and consider $\sigma\colon \Delta^i \to X$ a continuous map from an $i$-simplex $\Delta^i$ to $X$. Fix a vertex $u \in \Delta^i$, and let $x = \sigma(u)$.
	
	Since $h_1$ is an $r$-sheeted covering map, there are $\tilde{\sigma}_1, \dots, \tilde{\sigma}_r \colon \Delta^i \to Z$ distinct lifts of $\sigma$ by $h_1$, characterized by the property $\tilde{\sigma}_j(u) = z_j$. Then, 
	\[
	(h_1)_\ast f_1^\ast \varphi(\sigma) = \sum_j f_1^\ast \varphi (\tilde{\sigma}_j) = \sum_j \varphi(f_1 \circ \tilde{\sigma}_j).
	\]
    Similarly, let $y_1, \dots, y_r \in Y$ be the distinct lifts of $f_2(x)$, and consider $\tilde{\omega}_1, \dots, \tilde{\omega}_r\colon \Delta^i \to Y$ the distinct lifts of $f_2 \circ \sigma$ by $h_2$, characterized by the property $\tilde{\omega}_j(u) = y_j$. Then, 
	\[
	f_2^\ast (h_2)_\ast \varphi(\sigma) = (h_2)_\ast \varphi(f_2 \circ \sigma) = \sum_j  \varphi(\tilde{\omega}_j).
	\] 
	We now observe that we have the equality of sets $\{ f_1 \circ \tilde{\sigma}_j \}_j = \{ \tilde{\omega}_j \}_j$.
	Indeed, Condition (2) in the statement of the lemma implies that the simplices $\{f_1 \circ \tilde{\sigma}_j\}_j$ are all distinct, which implies that both sets have the same number of elements. Moreover, since $f_1 \circ \tilde{\sigma}_j$ is a lift of $f_2 \circ \sigma$ by $h_2$, we deduce the inclusion $\{f_1 \circ \tilde{\sigma}_j \} \subset \{ \tilde{\omega}_j \}_j$ and the desired equality of sets follows. 
    	
	From this equality of sets and the previous two calculations, we obtain the desired equality $(h_1)_\ast f_1^\ast = f_2^\ast (h_2)_\ast$ of cochain maps, which induces the result in cohomology.
\end{proof}
\begin{remark}
	Condition (2) of Lemma \ref{lemma: general lemma to commute pushforward and pullback} holds if the commutative diagram is Cartesian. 
\end{remark}

\begin{proposition}\label{prop: Eisenstein class compatible with pullback}
	Let $r, r' \in \Z$ with $r \geq r' \geq 1$, consider the projection map $\mathrm{pr}\colon T_r - T_r[c] \to T_{r'} - T_{r'}[c]$, and denote by $\mathrm{pr}^\ast$ the corresponding pullback in cohomology. Then, $\mathrm{pr}^\ast z_{r'} = z_r$.
\end{proposition}
\begin{proof}
    The structure of the proof is analogous to the proof of Lemma \ref{lemma: v_r*z_r is Gamma_r equivariant}, so we only outline the key points. First, we observe 
    \[
    \mathrm{pr}^\ast(T_{r'}[c] - c^n \{0 \} ) = T_{r}[c] - c^n \{ 0\} \in H^n(T_{r}, T_{r} - T_{r}[c]). 
    \]
    Therefore, $\mathrm{pr}^\ast(z_{r'})$ is a lift of $T_{r}[c] - c^n \{ 0\}$. Second, we claim that $\mathrm{pr}^\ast$ commutes with $[a]_\ast$. The key to proving this statement is to apply Lemma \ref{lemma: general lemma to commute pushforward and pullback} to the diagram
    \[
    \begin{tikzcd}
    {T_r - T_r[ac]} \arrow[r, "\mathrm{pr}"] \arrow[d, "{[a]}"] & {T_{r'} - T_{r'}[ac]} \arrow[d, "{[a]}"] \\
    {T_r - T_r[c]} \arrow[r, "\mathrm{pr}"]                     & {T_{r'} - T_{r'}[c]}.
    \end{tikzcd}
    \]
    Therefore, $z_r = \mathrm{pr}^\ast z_{r'}$ by Theorem \ref{thm: Eisenstein class of a torus bundle}.
\end{proof} 

From the previous proposition, we deduce that the classes $v_r^\ast z_r$ satisfy the following distribution relation.

\begin{proposition}\label{prop: pullback of Eis class trace compatible}
	Let $r \geq 1$ and consider the pushforward attached to the finite quotient map $\mathrm{pr}\colon \Gamma_{r+1}(q) \backslash \sX \to \Gamma_{r}(q) \backslash \sX$, namely
	\[
		\mathrm{pr}_\ast\colon H^{n-1}(\Gamma_{r+1}(q) \backslash \sX, \Z[1/c]) \too H^{n-1}(\Gamma_r(q) \backslash \sX, \Z[1/c]).
	\]
	Then, $\mathrm{pr}_\ast(v_{r+1}^\ast z_{r+1}) = v_r^\ast z_r$.
\end{proposition}
\begin{proof}
    Consider the map
    \[
    f_r\colon \Gamma_r(q) \backslash \sX \too T_r - T_r[c] \too T_1 - T_1[c],
    \]
    where the first arrow is induced by $v_r$ and the second one is the quotient map. Also, observe that since $r \geq 1$ we can define 
    \[
    f_{r+1}: \Gamma_{r+1}(q) \backslash \sX \too T_{r+1} - T_{r+1}[pc] \too T_1 - T_1[pc],
    \]
    in a similar way as $f_r$, but where we used that $v_{r+1}$ is of exact $p^{r+1}$ torsion, with $p^{r+1} > p$. It is a consequence of Proposition \ref{prop: Eisenstein class compatible with pullback} that, if $\iota \colon T_1 - T_1[pc] \to T_1 - T_1[c]$, 
    \[
    v_r^\ast z_r = f_r^\ast z_1, \quad v_{r+1}^\ast z_{r+1} =  f_{r+1}^\ast \iota^\ast z_{1}.
    \]
    We will now deduce the desired statement from the invariance of $z_1$ under multiplication by $p$. With this aim, observe that we can apply Lemma \ref{lemma: general lemma to commute pushforward and pullback} to the following commutative diagram 
    \[
    \begin{tikzcd}
    \Gamma_{r+1}(q) \backslash \sX \arrow[r, "f_{r+1}", hook] \arrow[d, "\mathrm{pr}"] & {T_1 - T_1[pc]} \arrow[d, "{[p]}"] \\
    \Gamma_r(q) \backslash \sX \arrow[r, "f_r", hook]                                  & {T_1 - T_1[c].}                    
    \end{tikzcd}
    \]
    Indeed, since $\Gamma_r(q)$ is torsion-free and $[\Gamma_r(q): \Gamma_{r+1}(q)] = p^n$, both horizontal maps are $p^n$-sheeted covering maps, implying Condition (1) of the lemma. Moreover, the fact that $\Gamma_r(q)$ is torsion-free implies that the maps $f_r$ and $f_{r+1}$ are injective, giving Condition (2) of the lemma. Therefore,
    \[
    \mathrm{pr}_\ast f_{r+1}^\ast = f_r^\ast [p]_\ast.
    \]
    From there, 
    \[
    \mathrm{pr}_\ast v_{r+1}^\ast z_{r+1} = \mathrm{pr}_\ast f_{r+1}^\ast \iota^\ast z_1 = f_r^\ast [p]_\ast \iota^\ast z_1 = f_r^\ast z_1 = v_r^\ast z_r,
    \]
    where we used the invariance of $z_1$ under multiplication by $p$ on the second to last equality (see Theorem \ref{thm: Eisenstein class of a torus bundle} and Remark \ref{rmk: details on pushforward by [a]}).
\end{proof}

\section{Differential form representative of the Eisenstein class}\label{sec: differential form representatives of Eisenstein class}
In \cite{BCG}, Bergeron, Charollois, and Garc\'ia construct a closed differential form on $T_r - T_r[c]$ representing the Eisenstein class $z_r$. Their construction, inspired by the work of Bismut and Cheeger \cite{BismutCheeger1992}, consists of a regularized average of a transgression form considered by Mathai and Quillen. In this section, we outline this procedure and use the differential forms we obtain to prove some properties about pullbacks of the Eisenstein class by torsion sections (see Proposition \ref{prop: average of pullbacks cEpsi by vectors of exact order p is 0}). The expressions given here will also be used in the last section to relate periods of the Eisenstein class to special values of $L$-functions.

\subsection{Mathai--Quillen form and the transgression form}\label{section: MQ form and the transgression form}

Let $S \coloneq \GL_n(\R) / \SO_n$ and consider the real vector bundle $E \coloneq S \times \R^n \to S$, which is $\GL_n(\R)$-equivariant for the left multiplication action on each of the components of $E$ and on $S$. Mathai and Quillen construct a closed $\GL_n(\R)$-equivariant differential form 
\[
\varphi \in \Omega_{\mathrm{rd}}^n(E)^{\GL_n(\R)}
\]
which has rapid decay (Gaussian shape) and integral $1$ along the fibers. In particular, $\varphi$ represents the Thom class of the oriented vector bundle $E \to S$ via the isomorphisms
\[
H^n(\Omega_{\mathrm{rd}}^\bullet(E)) \simeq H^n(E, E - \{0\}, \mathbb{R})
\]
between the cohomology of the complex of forms on $E$ with rapid decay along the fibers $\Omega_{\mathrm{rd}}^\bullet(E)$ and relative singular cohomology (see \cite[Page 98 and Page 99]{MQ}).

There is an explicit expression for the form $\varphi$, which we proceed to outline following \cite[Theorem 13]{BCG} and \cite{MQ}. The reader is referred to these sources for further details on the construction of $\varphi$, as for our purposes it is sufficient to know the shape of its expression. Using the Iwasawa decomposition of $\GL_n(\R)$, fix $h\colon S \to \GL_n(\R)$ a smooth section of the quotient map $\GL_n(\R) \twoheadrightarrow S$. Then
\begin{equation}\label{eq: explicit expression for varphi}
\varphi = \pi^{-n/2} e^{-\lvert h^{-1}x\rvert^2} \sum_{\substack{I \subset \{1, \dots, n\} \\ \lvert I \rvert \text{ even} }} \varepsilon_{I, I'} \mathrm{Pf}(\Omega_I/2) \left( d(h^{-1}x) + \theta h^{-1}x \right)^{I'},
\end{equation}
where:
\begin{itemize}
    \item $x = (x_1, \dots, x_n) \in \R^n$ and $\rvert x\lvert = (x_1^2 + \dots + x_n^2)^{1/2}$ is its standard norm.
    \item $\theta$ is an $n \times n$ matrix of $1$-forms on $S$, obtained as the pullback by $h$ of the connection of the principal $\SO_n$-bundle $\GL_n(\R) \to S$ given by $\theta_{\GL_n(\R)} = (g^{-1}dg - dg^t (g^t)^{-1})/2$.
    \item $\Omega$ is an $n\times n$ matrix of $2$-forms on $S$, obtained as the pullback by $h$ of the curvature $d\theta_{\GL_n(\R)} + \theta_{\GL_n(\R)}^2$. Then, $\mathrm{Pf}(\Omega_I/2)$ is an $\lvert I \rvert$-form given as the Pfaffian of the submatrix of $\Omega/2$ of size $\lvert I \rvert$ involving the indices in $I$.
    \item $I'$ denotes the complement of $I\subset \{1, \dots, n \}$, $\varepsilon_{I, I'} \in \{ \pm 1\}$, and for a vector $v$ of size $n$, $v^{I'} = v_{i_1} v_{i_2} \cdots v_{i_{\lvert I'\rvert}}$, where $I' = \{ i_1, \dots , i_{\lvert I' \rvert}\}$.
\end{itemize}
\begin{remark} 
We will not use the expressions for $\theta$ and $\Omega$, aside from the fact that they are forms of degree $1$ and $2$ on $S$.
\end{remark}

For $t \in \R_{>0}$, let $[t]\colon E \to E$ be multiplication by $t$ on the fibers. An important property of $\varphi$ is that for every $t \in \R_{>0}$, $[t]^\ast \varphi$ also represents the Thom class. Indeed, the Gaussian on the fibers gets dilated, but the value of the integral over the fibers is preserved and equal to $1$. In particular,
\[
[t]^\ast \varphi \too \delta_{0}, \ \text{ as } t \too +\infty,
\]
where $\delta_0$ denotes the current of integration along the zero section of $E$, also represents the Thom class (as a current). Recall that the Eisenstein class is a lift of Thom classes of a torus bundle, by Theorem \ref{thm: Eisenstein class of a torus bundle}. The next proposition constructs a form $\eta$ whose differential involves $\delta_{0}$. The relevance of this form is that a (regularized) average of it will give a representative of the Eisenstein class.

\begin{definition}
    Let $R \coloneq \sum_i x_i \frac{\partial}{\partial x_i}$ be the radial vector field on $E = S \times \R^n$, where $\{x_i \}_i$ denote the coordinates on $\R^n$ and consider the contraction $\psi \coloneq \iota_R \varphi \in \Omega_{\mathrm{rd}}^{n-1}(E)^{\GL_n(\R)}$, which is $\GL_n(\R)$-invariant (see \cite[Proposition 14]{BCG}).    
\end{definition}

\begin{proposition}\label{prop: transgression prop of psi}
    Consider the differential form on $E - S$
    \begin{equation}\label{eq: transgression of Thom form}
    \eta \coloneq \int_0^{+\infty} [t]^\ast \psi \frac{dt}{t}.
    \end{equation}
    Viewed as a current on $E$, it satisfies the transgression property $d \eta = \delta_0 - [0]^\ast \varphi$.
\end{proposition}
\begin{proof}
    The main idea for the proof of this statement lies in the following equalities
    \[
    \delta_0 - [0]^\ast \varphi = \int_0^{+\infty} \frac{d}{dt} [t]^\ast \varphi dt = \int_0^{+\infty} d [t]^\ast \iota_R \varphi \frac{dt}{t} = d \eta,
    \]
    where the second equality follows from interpreting $\frac{d}{dt} [t]^\ast \varphi$ in terms of a Lie derivative with respect to the vector field $R$ and using Cartan magic formula. For more details, see Section 7.2 and Section 7.3 of \cite{BCG} and Page 106 of \cite{MQ}.
\end{proof}

Using the explicit expression for $\varphi$ given in \eqref{eq: explicit expression for varphi}, and following the same notation as in that equation, we obtain
\[
\psi = \pi^{-n/2} e^{- \lvert h^{-1} x \rvert ^2} \sum_{\substack{I \subsetneq
 \{1, \dots, n\} \\ \lvert I \rvert \text{ even}}} \left(\varepsilon_{I, I'} \mathrm{Pf}(\Omega_I/2) \sum_{k = 1}^{\lvert I'\rvert} (-1)^{k+1} (h^{-1}x)_{i_k} \left( d(h^{-1}x) + \theta h^{-1}x \right)^{I' - \{ i_k\}}\right),
\]
\[
\eta = \frac{\pi^{-n/2}}{2} \sum_{\substack{I \subsetneq \{1, \dots, n\} \\ \lvert I \rvert \text{ even}}}  \left(\varepsilon_{I, I'} \mathrm{Pf}(\Omega_I/2) \frac{\Gamma(\lvert I' \rvert /2)}{\lvert h^{-1}x \rvert^{\lvert I' \rvert} } \sum_{k = 1}^{\lvert I'\rvert} (-1)^{k+1} (h^{-1}x)_{i_k} \left( d(h^{-1}x) + \theta h^{-1}x \right)^{I' - \{ i_k\}}\right).
\]
Here $I' = \{ i_1, \dots, i_{\lvert I' \rvert}\}$ is the complement of $I \subsetneq \{ 1, \dots, n \}$. The exact formulas will not be necessary for us. On the other hand, it will be important to note: 
\begin{itemize}
    \item $\varphi$ and $\psi$ are linear combinations of products of an exponential and a polynomial. In particular, they have rapid decay along the fibers. 
    \item $[0]^\ast \psi = 0$.
    \item $\eta$ does not have rapid decay along the fibers.
\end{itemize}

\subsection{Eisenstein transgression}
We proceed to consider a regularized average of the form $\eta$ in \eqref{eq: transgression of Thom form} over a lattice to obtain forms on torus bundles representing the Eisenstein class. For $L \subset \Q^n$ a $\Z$-lattice and $\lambda \in L$, let
\[
\mathrm{tr}_\lambda \colon E \too E , \ (g,x) \mapstoo (g , x + \lambda).
\]
Then, if $t \in \R_{>0}$, define 
\begin{equation}\label{eq: det theta(psi,L)}
	\theta([t]^\ast \psi, L) \coloneq \sum_{\lambda \in L}\mathrm{tr}_\lambda^\ast [t]^\ast \psi.
\end{equation}
The sum converges as the differential form $t^\ast \psi$ has rapid decay on the fibers of $E \to S$. 
\begin{theorem}\label{thm: meromorphic continuation of Epsi}
    View $\theta([t]^\ast \psi, L)$ as a differential form on $S \times (\R^n - L)$. For $s \in \C$ with $\mathrm{Re}(s) \gg 0$, the integral
    \[
    E_\psi(L, s) \coloneq \int_0^{+\infty} \theta([t]^\ast \psi, L)t^s \frac{dt}{t}
    \]
    converges. Furthermore, it admits a meromorphic continuation to all $s \in \C$, regular at $s = 0$, and its value at every regular $s \in \C$ defines a differential form on $S \times (\R^n - L)$. 
\end{theorem}
\begin{proof}
    This follows from Proposition 17 and Section 8.5 of \cite{BCG}. In particular, the fact that the integral is regular at $s = 0$ follows from the fact that we are viewing $\theta([t]^\ast \psi, L)$ as a form on $S \times (\R^n - L)$, and $[t]^\ast \psi$ tends to $0$ as $t \to +\infty$ on $S \times (\R^n - L)$.
\end{proof}

The previous theorem implies that $E_\psi(L,s)$ is regular at $s = 0$ and 
\[
E_\psi(L) \coloneq E_\psi(L,0)
\]
defines a form on $S \times (\R^n - L)/L$. In fact, $E_\psi(L)$ descends to a form in $\sX \times (\R^n-L)/L$ by the calculation on (8.9) of \cite{BCG}. Moreover, if $\Gamma' \subset \SL_n(\R)$ is a subgroup contained in the stabilizer of $L$, the form $E_\psi(L)$ is invariant under $\Gamma'$.

\begin{remark}\label{rmk: Epsi as a regularized average}
    We outline how to view $E_\psi(L)$ as
    a regularized average of $\eta$. As we pointed out at the end of Section \ref{section: MQ form and the transgression form}, the form $\eta$ does not have rapid decay along the fibers. Therefore, the sum $\sum_{\lambda \in L} \mathrm{tr}_\lambda^\ast \eta$
    does not converge. On the other hand, for $s \in \C$ with $\mathrm{Re}(s) \gg 0$ define  
    \[
    \eta(s) \coloneq \int_0^{+\infty} [t]^\ast \psi t^s \frac{dt}{t}.
    \]
    Then, $\eta(s)$ has the same expression as the one given for $\eta$ at the end of Section \ref{section: MQ form and the transgression form} where the term $\Gamma(\lvert I' \rvert /2)/\lvert h^{-1}x \rvert^{\lvert I' \rvert}$ is replaced by $\Gamma((\lvert I' \rvert + s) /2)/\lvert h^{-1}x \rvert^{\lvert I' \rvert + s}$. In particular, it follows that if $\mathrm{Re}(s) \gg 0$, the sum $\sum_{\lambda \in L} \mathrm{tr}_\lambda^\ast \eta(s)$ is absolutely convergent and 
    \[
    E_\psi(L, s) = \int_0^{+\infty} \theta([t]^\ast \psi, L)t^s \frac{dt}{t} = \sum_{\lambda \in L} \mathrm{tr}_\lambda^\ast \eta(s), 
    \]
    where we exchanged the integral with the sum (using that for $\mathrm{Re}(s) \gg 0$, the sums are absolutely convergent). Thus, $E_\psi(L)$ is equal to the value at $s = 0$ of the meromorphic continuation of $\sum_{\lambda \in L} \mathrm{tr}_\lambda^\ast \eta(s)$. 
\end{remark}

Recall the torus bundle 
\[
T_r = \Gamma_r(q) \backslash (\sX \times \R^n / \Z^n) \too \Gamma_r(q) \backslash \sX
\]
introduced in Section \ref{subsec: Eisenstein class of universal family of tori with level structure}. 
\begin{definition}
    Consider the linear combination
    \[
    {}_c E_\psi \coloneq E_\psi(c^{-1}\Z^n) - c^nE_\psi(\Z^n),
    \]
    which we view as a differential form on $T_r - T_r[c]$ for every $r \geq 1$.
\end{definition} 

\begin{theorem}\label{thm: differential form describing the Eisenstein class}
	The form ${}_c E_\psi$ is closed in $T_r - T_r[c]$ and its cohomology class 
    \[
    {}_c E_\psi \in H^{n-1}_{\mathrm{dR}}({T}_r - {T}_r[c]) \simeq H^{n-1}({T}_r - {T}_r[c], \R)
    \]
    is equal to the image of the Eisenstein class $z_r$ in $H^{n-1}({T}_r - {T}_r[c] , \R)$.		
\end{theorem}
\begin{proof}
    See Theorem 19, Proposition 20, and Theorem 21 of \cite{BCG}. There it is explained that, since $E_\psi$ is a regularized average of $\eta$ (see Remark \ref{rmk: Epsi as a regularized average}), Proposition \ref{prop: transgression prop of psi} implies that 
    \[
    d({}_c E_\psi) = \delta_{T[c]} - c^n \delta_{\{ 0\}},
    \]
    where $\delta_{T[c]}$ and $\delta_{\{ 0\}}$ denote currents of integration along $T[c]$ and $\{ 0\}$ (the contributions $[0]^\ast \varphi$ appearing in Proposition \ref{prop: transgression prop of psi} vanish after the regularization). Moreover, $[a]_\ast E_\psi = E_\psi$ by Proposition 20 of \cite{BCG}. Thus, ${}_c E_\psi$ is a closed form on $T_r - T_r[c]$ satisfying the characterizing properties of the Eisenstein class $z_r$ asserted in Theorem \ref{thm: Eisenstein class of a torus bundle}.
\end{proof}

\subsection{Pullbacks by torsion sections}\label{subsec: pullbacks by torsion sections}
We now use the differential forms introduced above to study the pullbacks of the form ${}_c E_\psi$ by torsion sections. For $v \in \Q^n$, denote also by $v$ the corresponding section $v\colon S \to E$. Then, for $s \in \C$ with $\mathrm{Re}(s) \gg 0$, consider the differential form on $S$
\[
\eta(v, s) \coloneq \int_0^{+\infty} v^\ast [t]^\ast\psi t^s \frac{dt}{t} = \int_{0}^{+\infty} (tv)^\ast \psi t^s \frac{dt}{t}.
\]
Since $0^\ast \psi= 0$, which can be verified using the explicit expression given at the end of Section \ref{section: MQ form and the transgression form}, we have $\eta(0, s) = 0$. From this same expression and Remark \ref{rmk: Epsi as a regularized average}, we deduce that for $v \neq 0$
\begin{equation}\label{eq: expression eta(v,s)}
\begin{split}
&\eta(v,s) = \\ &\frac{\pi^{-n/2}}{2} \sum_{\substack{I \subsetneq \{1, \dots, n\} \\ \lvert I \rvert \text{ even}}} \left(\varepsilon_{I, I'} \mathrm{Pf}(\Omega_I/2) \frac{\Gamma((\lvert I' \rvert + s)/2)}{\lvert h^{-1}v \rvert^{\lvert I' \rvert + s} } \sum_{k = 1}^{\lvert I' \rvert} (-1)^{k+1} (h^{-1}v)_{i_k} \left( d(h^{-1})v + \theta h^{-1}v \right)^{I' - \{ i_k\}}\right).
\end{split}
\end{equation}
\begin{proposition}\label{prop: meromorphic continuation of sum of eta}
	Let $L \subset \Q^n$ be a $\Z$-lattice and $v \in \Q^n - L$. For $s \in \C$ with $\mathrm{Re}(s) \gg 0$,
	\[
	v^\ast E_{\psi}(L, s) = \sum_{\lambda \in v + L} \eta(\lambda , s).
	\]
	In particular, the right-hand side has a meromorphic continuation regular at $s = 0$.
\end{proposition}
\begin{proof}
    This follows from Theorem \ref{thm: meromorphic continuation of Epsi} and Remark \ref{rmk: Epsi as a regularized average}.
\end{proof}
\noindent Thus, if $v \in \Q^n - (1/c)\Z^n$,
\begin{equation}\label{eq: v Epsi in terms of eta}
	v^\ast {}_c E_\psi = \lim_{s \to 0} \sum_{\lambda \in v + c^{-1} \Z^n} \eta(\lambda, s) - c^n \sum_{ \lambda \in v + \Z^n} \eta(\lambda, s),
\end{equation}
where here and from now on, $\lim_{s \to 0}$ denotes evaluation of the meromorphic continuation. 

In fact, the right-hand side of the equation appearing in Proposition \ref{prop: meromorphic continuation of sum of eta} defines a differential form on $S$ even if $v \in L$. More precisely, 
\[
\sum_{\lambda \in L} \eta(\lambda, s)
\]
converges for $\mathrm{Re}(s) \gg 0$, and admits a meromorphic continuation to $\C$ which is regular at $s = 0$. We proceed to prove a weaker version of this statement, as this will be enough for our purposes.

\begin{lemma}\label{lemma: pullback at 0}
	Let $g \in S$, consider tangent vectors $Y_1, \dots, Y_{n-1} \in T_g S$ and denote $Y = (Y_1, \dots, Y_{n-1})$. Then, for $s \in \C$ with $\mathrm{Re}(s) \gg 0$, 
	\[
	s \mapstoo \sum_{\lambda \in L} \eta(\lambda, s)_{g}(Y)
	\]
	converges and admits a meromorphic continuation to $\C$ which is regular at $s = 0$. 
\end{lemma}
\begin{proof}
	It follows from the explicit expression of $\eta(v,s)$ given in \eqref{eq: expression eta(v,s)} that the sum 
    \[
    \sum_{\lambda \in L} \eta(\lambda, s)_{g}(Y)
    \]
    is absolutely convergent for $\mathrm{Re}(s) \gg 0$. From there, we deduce that if $\mathrm{Re}(s) \gg 0$, we have the equality 
	\[
	\sum_{\lambda \in L} \eta(\lambda, s)_{g}(Y) = \int_0^{+\infty} \sum_{\lambda \in L} \left((t\lambda)^\ast \psi\right)_{g}(Y) t^s \frac{dt}{t},
	\]
    as we can exchange the integral with the sum.  
	Thus, it is enough to prove that the right-hand side has a meromorphic continuation regular at $s = 0$. For that, define the function 
	\[
	f\colon \R^n \too \R, \ v\mapstoo (v^\ast \psi)_{g}(Y).
	\]
	Since $\psi$ is a differential form which has rapid decay along the fibers, it follows that $f$ is a Schwartz function. Hence, we need to prove that 
	\begin{equation}\label{eq: function we want continuation}
	\int_{0}^{+\infty} \sum_{\lambda \in L} f(t\lambda) t^s \frac{dt}{t}
	\end{equation}
	has a meromorphic continuation to $s \in \C$ which is regular at $s = 0$. We split the integral as a sum of integrals from $1$ to $+\infty$ and from $0$ to $1$. Observe that $f(0) = 0$, as $0^\ast \psi = 0$. The rapid decay of $f$, together with the fact that $f(0) = 0$, implies that the integral from $1$ to $+\infty$ converges absolutely and defines an entire function on $s$. To study the integral from $0$ to $1$, we use Poisson summation formula
	\[
	\int_0^1 \sum_{\lambda \in L} f(t\lambda) t^s \frac{dt}{t} = \int_{0}^1 \sum_{\lambda \in L^{\vee}} \hat{f}(\lambda/t) t^{s - n}\frac{dt}{t},
	\]  
	where $\hat{f}$ denotes the Fourier transform of $f$ and $L^{\vee}$ the dual lattice of $L$. For $\mathrm{Re}(s) \gg n$, the previous integral can be written as 
	\[
	\frac{\hat{f}(0)}{s-n} + \int_1^{+\infty} \sum_{\lambda \in L^\vee - \{ 0 \}} \hat{f}(\lambda u) u^{n-s} \frac{du}{u}.
	\]
	Since $\hat{f}$ is a Schwartz function, the integral converges for all values of $s \in \C$ and defines an entire function. Thus, this expression gives a meromorphic continuation of the integral from $0$ to $1$ regular everywhere except maybe at $s = n$. The result follows from there.	
\end{proof}

Finally, we are ready to prove the following expression regarding pullbacks of the Eisenstein class by torsion sections, which will be useful for the next section.

\begin{proposition}\label{prop: average of pullbacks cEpsi by vectors of exact order p is 0}
	For $v \in \Q^n - c^{-1}\Z^n$, view  $v^\ast {}_c E_\psi$ as a differential form on $\sX$. Then, 
	\[
	\sum_{v \in \frac{1}{p}\Z^n / \Z^n - \{ 0\}} v^\ast {}_c E_\psi = 0.
	\]
\end{proposition}
\begin{proof}
	By Proposition \ref{prop: meromorphic continuation of sum of eta}, and more precisely \eqref{eq: v Epsi in terms of eta}, we can write the sum of the proposition as the evaluation at $s = 0$ of the following expression
	\[
	\sum_{v \in \frac{1}{p}\Z^n/\Z^n - \{0 \} } \sum_{\lambda \in v + c^{-1}\Z^n} \eta(\lambda, s) - c^n \sum_{v \in \frac{1}{p}\Z^n/\Z^n - \{0 \}} \sum_{\lambda \in v + \Z^n} \eta(\lambda, s).
	\]
	We will verify that each of the two terms vanishes when evaluated at $s = 0$. Since the proof is analogous in the two cases, we will show that  
	\[
	\lim_{s \to 0}\sum_{v \in \frac{1}{p}\Z^n/\Z^n - \{0 \}} \sum_{\lambda \in v + \Z^n} \eta(\lambda, s) = 0.
	\]

    Let $g \in S$, consider tangent vectors $Y_1, \dots, Y_{n-1} \in T_g S$, and let $Y = (Y_1, \dots, Y_{n-1})$. Then, it is enough to see
    \[
    \lim_{s \to 0} \sum_{v \in \frac{1}{p}\Z^n/\Z^n - \{0 \}} \sum_{\lambda \in v + \Z^n} \eta(\lambda, s)_g(Y) = 0.
    \]
    Then, for $s \in \C$ with $\mathrm{Re}(s) \gg 0$
	\[
    \begin{split}
	\sum_{v \in \frac{1}{p}\Z^n/\Z^n - \{0 \}} \sum_{\lambda \in v + \Z^n} \eta(\lambda, s)_g(Y)  = & \sum_{v\in \frac{1}{p}\Z^n} \eta(\lambda, s)_g(Y) - \sum_{\lambda \in \Z^n} \eta(\lambda,s)_g(Y),
    \end{split}
	\]
	where the right-hand side consists of the difference of two functions which admit a meromorphic continuation to all $s \in \C$ and are regular at $s = 0$ by Lemma \ref{lemma: pullback at 0}. The previous expression is equal to
	\[
	\sum_{\lambda \in \Z^n} \eta(\lambda/p, s)_g(Y) - \sum_{\lambda \in \Z^n} \eta(\lambda, s)_g(Y) = p^s\sum_{\lambda \in \Z^n} \eta(\lambda, s)_g(Y) - \sum_{\lambda \in \Z^n} \eta(\lambda, s)_g(Y).
	\]
	Here we used that $\eta(\lambda/p, s) = p^s \eta(\lambda, s)$, which can be verified from the definition of $\eta(v,s)$.
	Since the meromorphic continuation of  $\sum_{\lambda \in \Z^n} \eta(\lambda, s)_g(Y)$ is regular at $s = 0$ by Lemma \ref{lemma: pullback at 0}, the evaluation at $s = 0$ of the expression above is zero.
\end{proof}

\section{The Eisenstein group cohomology class}\label{sec: Topological construction of the Eisenstein group cohomology class}

In this section, we package the pullbacks of the Eisenstein class by $p$-power torsion sections in a group cohomology class for $\Gamma \coloneq \SL_n(\Z)$ valued in measures on $\X\coloneq \Z_p^n - p\Z_p^n$. Then, we discuss the process of lifting this class to a class valued in total mass zero measures on $\X$, which will be an important property for defining rigid classes and $p$-adic invariants attached to totally real fields.

\subsection{From singular to group cohomology} 
Let $r \geq 1$ and let
\[
m \coloneq \mathrm{lcm}(c, [\Gamma: \Gamma(q)] , 2).
\]
Since $\Gamma_r(q)$ is normal in $\Gamma_r$, there are actions of $\Gamma_r$ on the singular cohomology of $\Gamma_r(q) \backslash \sX$, described above Lemma \ref{lemma: v_r*z_r is Gamma_r equivariant}, and on the group cohomology of $\Gamma_r(q)$, via conjugation. These actions are compatible with the natural isomorphism from singular to group cohomology, giving 
\begin{equation}\label{eq: iso singular coho with group coho + restr group coho}
H^{n-1}(\Gamma_r(q) \backslash \sX, \Z[1/m])^{\Gamma_r} \xlongrightarrow{\sim} H^{n-1}(\Gamma_r(q), \Z[1/m])^{\Gamma_r} \xlongrightarrow{\sim} H^{n-1}(\Gamma_r, \Z[1/m]).
\end{equation}
The first map is an isomorphism because $\Gamma_r(q)$ acts freely on $\sX$, as it is torsion-free. The second one is given by the corestriction map multiplied by $[\Gamma_r : \Gamma_r(q)]^{-1}$, which belongs to $\Z[1/m]$ as $[\Gamma_r : \Gamma_r(q)]$ divides $[\Gamma: \Gamma(q)]$. The inverse of the second map is restriction.

For every $r \geq 1$, in Section \ref{subsec: topo construction of Eis class} we constructed the classes
\[
v_r^\ast z_r \in H^{n-1}(\Gamma_r(q) \backslash \sX, \Z[1/m])^{\Gamma_r, -}.
\]
and proved they are invariant under the action of $\Gamma_r$ and belong to the $-1$-eigenspace for the action induced by $w \coloneq \mathrm{diag}(1, -1, 1, \dots, 1) \in \GL_n(\Z)$ in Lemma \ref{lemma: v_r*z_r is Gamma_r equivariant} and Lemma \ref{lemma: action of w on pullback of Eis class}.
\begin{definition}\label{def: definition of c_r}
	For $r \geq 1$, let $c_r \in H^{n-1}(\Gamma_r, \Z[1/m])$ be the group cohomology class corresponding to $v_r^\ast z_r$ via the isomorphisms of \eqref{eq: iso singular coho with group coho + restr group coho}. 
\end{definition}
Similarly as above, $w \in \GL_n(\Z)$ induces an action on group cohomology for $\Gamma_r(q)$ (as well as for $\Gamma_r$) via conjugation. This is compatible with the involution in singular cohomology induced by $w$ considered in Section \ref{subsec: Eisenstein class of universal family of tori with level structure} via \eqref{eq: iso singular coho with group coho + restr group coho}. It follows that $c_r \in H^{n-1}(\Gamma_r, \Z[1/m])^{-}$, where here and form now on the superindex $-$ indicates the $-1$-eigenspace for $w$. 

The trace compatibility of the singular cohomology classes $(v_r^\ast z_r)_r$ leads to the compatibility of the group cohomology classes $(c_r)_r$ with respect to corestriction maps.
\begin{proposition}\label{prop: trace compatibility c_r}
	For $r \geq 1$ let $\mathrm{cor}\colon H^{n-1}(\Gamma_{r+1}, \Z[1/m]) \to H^{n-1}(\Gamma_r, \Z[1/m])$ be the corestriction map. Then, $\mathrm{cor}(c_{r+1}) = c_r$.
\end{proposition}
\begin{proof}
	Denote by $c_r(q) \in H^{n-1}(\Gamma_{r}(q) , \Z[1/m])^{\Gamma_r}$ the image of $v_r^\ast z_r$ via the first isomorphism in \eqref{eq: iso singular coho with group coho + restr group coho}. Since this isomorphism is compatible with respect to pushforward and corestriction (see \cite[Chapter III, Section 9 (E)]{browncohomology}), it follows from Proposition \ref{prop: pullback of Eis class trace compatible} that if 
	\[
	\mathrm{cor}_q \colon H^{n-1}(\Gamma_{r+1}(q), \Z[1/m]) \too H^{n-1}(\Gamma_r(q), \Z[1/m]),
	\] 
    denotes corestriction in group cohomology, then $\mathrm{cor}_q(c_{r+1}(q)) = c_r(q)$. This implies that 
    \[
    \mathrm{cor}\left([\Gamma_{r+1}:\Gamma_{r+1}(q)] c_{r+1} \right)  = [\Gamma_{r}:\Gamma_{r}(q)] c_{r}.
    \]
    which leads to the desired result as $[\Gamma_{r+1}: \Gamma_{r+1}(q)] = [\Gamma_r : \Gamma_r(q)] \in \Z[1/m]^\times$.
\end{proof}
It is a computation to verify that the corestriction maps are equivariant with respect to the involution $w$. From there, we conclude
\[
(c_r)_r \in \varprojlim_r H^{n-1}(\Gamma_r, \Z[1/m])^-.
\]
\subsection{Cohomology class with coefficients in $\Z[1/m]$-measures} We first describe the action of $w \in \GL_n(\Z)$ on group cohomology with coefficients. For every $r \geq 0$ (including $\Gamma_0 = \Gamma$), conjugation by $w$ induces the automorphism $\alpha\colon \Gamma_r \to \Gamma_r$, $\alpha(\gamma) = w \gamma w$. Then, if $M$ is a $\GL_n(\Z)$-module $M$, we can consider the morphism of complexes of group cochains
\[
C^\bullet(\Gamma_r, M) \too C^\bullet(\Gamma_r, M), \ c \mapstoo w \circ c \circ \alpha^r, 
\]
which induces an involution $w$ on $H^i(\Gamma_r, M)$. We will denote by $H^i(\Gamma_r, M)^{-}$ the $(-1)$-eigenspace for $w$.

For $r \geq 1$, let $\X_r \coloneq (\Z/p^r\Z)^n - (p\Z/p^r\Z)^n$ and if $A$ is an abelian group, denote 
\[
\D(\X_r, A) \coloneq \mathrm{Maps}(\X_r, A).
\]
It admits a left action of $\GL_n(\Z)$ given by $(g \cdot \lambda)(x) = \lambda(g^{-1}x)$, for $g \in \GL_n(\Z)$, $\lambda \in \D(\X_r, A)$, and $x \in \X_r$. Let $x_r \coloneq (1, 0, \dots, 0)^t \in \X_r$. Since the stabilizer of $x_r$ in $\Gamma$ is $\Gamma_r$, we deduce that we have a $\Gamma$-equivariant isomorphism
\[
\coInd_{\Gamma_r}^{\Gamma}(A) \xlongrightarrow{\sim} \D(\X_r, A ), \ f \longmapsto \lambda_f,
\]
where $\lambda_f(x) = f(\gamma)$ for $\gamma \in \Gamma$ such that $\gamma x_r = x$. In particular, Shapiro's lemma induces an isomorphism
\begin{equation}\label{eq: iso from Shapiro's lemma}
H^i(\Gamma, \D(\X_r, \Z[1/m]) \xlongrightarrow{\sim} H^i(\Gamma_r, \Z[1/m]), \ [\lambda] \mapstoo [c(\lambda)]
\end{equation}
where $c(\lambda)(\gamma_0, \dots, \gamma_{i}) = \lambda(\gamma_0, \dots, \gamma_i)(x_r)$. Moreover, the isomorphism is equivariant with respect to the action of $w$. 
\begin{definition}\label{def: mu_r}
    For every $r \geq 1$, define $\mu_r \in H^{n-1}(\Gamma, \D(\X_r, \Z[1/m]))^-$ to be the image of $c_r$ by the inverse of the isomorphism induced by Shapiro's lemma given in \eqref{eq: iso from Shapiro's lemma}.
\end{definition}
Consider the $\GL_n(\Z)$-equivariant maps
\begin{equation}\label{eq: transition map ur}
u_{r+1} \colon \D(\X_{r+1}, A) \too \D(\X_{r} , A), \ u_{r+1}(f)(x) = \sum_{\substack{ x' \in \X_{r+1} \\ x' \equiv x \mod p^{r} }}  f(x').
\end{equation}
It follows from the compatibility of the classes $(c_r)_r \in \varprojlim_r H^{n-1}(\Gamma_r , \Z[1/m])^-$ proven in Proposition \ref{prop: trace compatibility c_r}, that we have a compatible system 
\[
(\mu_r)_r \in \varprojlim_r H^{n-1}(\Gamma, \D(\X_r, \Z[1/m]))^-,
\]
where the transition maps are given by $u_r$ for every $r \geq 2$. This statement can be proven using Chapter III, Section 9 (A) of \cite{browncohomology}, which leads to describing the corestriction maps in terms of the map given by Shapiro's lemma and $u_r$.

Denote by $\D(\X, A)$ the space of $A$-valued distributions on $\X$. An element of $\lambda \in \D(\X, A)$ is determined by the values $\lambda(U)$ of the characteristic functions of compact open sets $U$. In particular, it is determined by the images of the following compact open sets. For $x \in \X_r$, choose any lift of it in $\X$, also denoted by $x$, and let 
\begin{equation}\label{eq: def of U_x/p^r}
U_{x/p^r} \coloneq x + p^r \Z_p^n \subset \X. 
\end{equation}
Endow $\D(\X, A)$ with a left action of $\GL_n(\Z)$ given by $(g \cdot \lambda)(U) = \lambda(g^{-1}U)$ and define 
\[
 \D(\X, A) \ontoo \D(\X_r, A), \ \lambda \mapstoo \lambda_r,
\]
where for $x \in \X_r$, $\lambda_r(x) = \lambda(U_{x/p^r})$. This discussion implies the following lemma.
\begin{lemma}\label{lemma: inverse limit leads to distributions on X}
	Let $A$ be an abelian group. The map 
	\[
	\D(\X, A) \xlongrightarrow{\sim} \varprojlim_{r} \D(\X_r, A), \ \lambda \mapstoo (\lambda_r)_r,
	\]
	is a $\Gamma$-equivariant isomorphism.
\end{lemma}

We will now combine the compatible system of classes $(\mu_r)_r$ to a group cohomology class valued on $\D(\X, \Z[1/m])$. First, we note the following fact regarding the cohomology of $\Gamma = \Gamma_0$ and the congruence subgroups $\Gamma_r$ for $r \geq 1$ in the stable range, which is a consequence of the work of Li and Sun \cite{LS}.

\begin{lemma}\label{lemma: H^i(Gamma_r , Z[1/m])^- finite for i <= n-2}
    For every $r \geq 0$ and $0 \leq i \leq n - 2$, the group $H^i(\Gamma_r, \Z[1/m])^{-}$ is finite. 
\end{lemma}
\begin{proof}
    Denote by $\tilde{\Gamma}_r$ the stabilizer of $v_r \in \Q^n/\Z^n$ in $\GL_n(\Z)$. Then, Shapiro's lemma implies the following isomorphisms.
    \[
    H^i(\Gamma_r , \R)^- \simeq H^i(\tilde{\Gamma}_r, \R(\det)) \simeq H^i(\GL_n(\Z), I), 
    \]
    where $I \coloneq \coInd_{\tilde{\Gamma}_r}^{\GL_n(\Z)}(\R(\det))$. Since $I^{\GL_n(\Z)} = 0$, again by Shapiro's lemma, it follows from Example 1.10 of \cite{LS} that 
    \[
    H^i(\Gamma_r, \R)^-= 0.
    \]
    It is now an application of the universal coefficient theorem that $H^i(\Gamma_r, \Z[1/m])^-$ is torsion. 

    By \cite[Theorem 11.4]{BorelSerre}, the group $\Gamma_r$ is of type (WFL). In particular, it is of type (VFL). By the Remark in Page 101 of Section 1.8 of \cite{SerreCohomologiedesgroupes}, and the universal coefficient theorem, it follows that $H^i(\Gamma_r, \Z[1/m])^-$ is finitely generated over $\Z[1/m]$. Since it is also a torsion group, we deduce that it is finite, as desired. 
\end{proof}

\begin{proposition}\label{prop: class valued on measures in terms of inverse limit of classes}
    For every $0 \leq i \leq n-1$, the map $\lambda \mapsto (\lambda_r)_r$ of Lemma \ref{lemma: inverse limit leads to distributions on X} induces an isomorphism
    \[
    H^{i}\left(\Gamma, \D(\X, \Z[1/m])\right)^{-} \xlongrightarrow{\sim} \varprojlim_r H^{i}(\Gamma , \D(\X_r, \Z[1/m]))^-.
    \]
\end{proposition}
\begin{proof}
    To simplify the notation, denote $\D \coloneq \D(\X, \Z[1/m])$ and $\D_r \coloneq \D(\X_r, \Z[1/m])$.  For a group $G$, a $G$-module $M$, and $j \in \Z_{\geq 0}$, let
\[
C^j(G, M) \coloneq \Hom_{G}(\Z[G^{j+1}], M),
\]
where the action of $G$ in $G^{j+1}$ is diagonal. The complex $C^\bullet(G, M)$ with the usual coboundary maps computes the group cohomology of $G$ with coefficients in $M$.

The surjective morphisms $u_r \mod p^{r-1}\colon \D_r \to \D_{r-1}$ obtained by taking the maps in \eqref{eq: transition map ur} modulo $p^{r-1}$ induce surjective maps $u_r\colon C^i(\Gamma, \D_r ) \twoheadrightarrow C^i(\Gamma, \D_{r-1})$ which can be combined as
\[
u = (u_r)\colon \prod_{r \geq 1} C^i(\Gamma, \D_r ) \ontoo \prod_{r \geq 1} C^i(\Gamma, \D_r).
\]
Denote by $1$ the identity map. 
Since $u_r$ is surjective for every $r$, we deduce that $1 - u$ is also surjective, as it is possible to find the preimage of any tuple via a recurrence. In particular, we have a short exact sequence of complexes 
\[
0 \too C^{\bullet}(\Gamma, \D) \too \prod_{r \geq 1} C^\bullet(\Gamma , \D_r) \xlongrightarrow{ 1 - u} \prod_{r \geq 1}C^\bullet(\Gamma, \D_r) \too 0.
\]
Note that to justify exactness in the middle, we used that $\D= \varprojlim_r \D_r$ by Lemma \ref{lemma: inverse limit leads to distributions on X}. Since $2$ is invertible in $\Z[1/m]$, we can consider the $w = -1$ eigenspace of the corresponding long exact sequence in cohomology. This yields the short exact sequence
\begin{equation}\label{eq: ses involving R1lim}
	0 \too R^1\varprojlim_r  H^{i-1}(\Gamma, \D_r)^-  \too H^i(\Gamma, \D)^- \too \varprojlim_r H^i(\Gamma, \D_r)^- \too 0,
\end{equation}
where we used that 
\[
R^1 \varprojlim_r H^{i-1}(\Gamma, \D_r)^- = \mathrm{coker}\left( \prod_{r \geq 1} H^{i-1}(\Gamma, \D_r)^- \xrightarrow{1 - u} \prod_{r \geq 1} H^{i-1}(\Gamma, \D_r)^- \right).
\]
Finally, since $H^{i-1}(\Gamma, \D_r)^- \simeq H^{i-1}(\Gamma_r, \Z[1/m])^-$ is finite for $i-1 \leq n-2$ by Lemma \ref{lemma: H^i(Gamma_r , Z[1/m])^- finite for i <= n-2}, it follows that $\left( H^{i-1}(\Gamma, \D_r)^-\right)_r$ satisfies the Mittag--Leffler condition. Thus, $R^1\varprojlim_r H^{i-1}(\Gamma, \D_r)^- = 0$ for every $i-1 \leq n - 2$, proving the desired isomorphism.
\end{proof}

\begin{definition}\label{def: muZp}
    Define 
    \[
    \mu \in H^{n-1}(\Gamma, \D(\X, \Z[1/m]))^-
    \]
    to be the class corresponding to $(\mu_r)_r$ via the isomorphism of Proposition \ref{prop: class valued on measures in terms of inverse limit of classes}.
\end{definition}
\begin{remark}\label{rmk: relation to BKL}
    The class $\mu$ viewed as a class with coefficients in $\Z_p$-valued measures on $\X$ is equal to the restriction of the classes considered in \cite[Definition 1.8.4]{beilinson2018} to measures on primitive vectors on $\Z_p^n$.
\end{remark}
\subsection{Cocycle with coefficients in $\R$-distributions}

Using the differential form ${}_c E_\psi$ introduced in Section \ref{sec: differential form representatives of Eisenstein class}, which represents the Eisenstein class, we give an explicit representative of the image of $\mu_r \in H^{n-1}(\Gamma, \D(\X_r, \Z[1/m]))$ in $H^{n-1}(\Gamma, \D(\X_r, \R))$. This will be used to lift $\mu$ to a class valued in measures of total mass zero and to compare our constructions to special values of $L$-functions.

\begin{lemma}\label{lemma: representative of c_r over R}
	Let $r \geq 1$ and let $z \in \sX$ be an arbitrary point. The map
	\[
	c_{v_r^\ast {}_c E_\psi}\colon \Gamma_r^n \too \R, \ (\gamma_0, \dots, \gamma_{n-1}) \mapstoo \int_{\Delta(\gamma_0 z, \dots, \gamma_{n-1}z)} v_r^\ast {}_c E_\psi,
	\]
	where $\Delta(\gamma_0z , \dots , \gamma_{n-1}z)$ denotes the geodesic simplex in $\sX$ with vertices $\{ \gamma_i z \}_{i}$, defines a group cocycle and represents the class $c_r \in H^{n-1}(\Gamma_r , \R)$.
\end{lemma}
\begin{proof}
	The form $v_r^\ast {}_c E_\psi$ on $\sX$ is closed and invariant under the action of $\Gamma_r$. It follows from there that $c_{v_r^\ast {}_c E_\psi}$ is a group cocycle and its cohomology class is independent of the choice of point $z \in \sX$.
	
	We proceed to see that the class of $c_{v_r^\ast {}_c E_\psi}$ is $c_r$. For this, note that Theorem \ref{thm: differential form describing the Eisenstein class} implies that $v_r^\ast {}_c E_\psi$ descends to a  closed differential form on $\Gamma_r(q) \backslash \sX$ representing $v_r^\ast z_r \in H^{n-1}(\Gamma_r(q)\backslash \sX, \R)$. Thus, the image of $v_r^\ast z_r$ by the first map in the isomorphism \eqref{eq: iso singular coho with group coho + restr group coho} (with coefficients in $\R$) is represented by the restriction of $c_{v_r^\ast {}_c E_\psi}$ to $\Gamma_r(q)^n$. In particular, it follows from the definition of $c_r$ that $[c_{v_r^\ast {}_c E_\psi}] = c_r$. 
\end{proof}
Fix $z \in \sX$ an arbitrary point. Define a cocycle 
\[
\mu_{v_r^\ast {}_c E_\psi}\colon \Gamma^n \too \D(\X_r, \R), \ (\gamma_0, \dots, \gamma_{n-1}) \mapstoo \left( \bar{x} \mapstoo  \int_{\Delta(\gamma_0 z, \dots, \gamma_{n-1} z)}  (x/p^r)^\ast {}_c E_{\psi} \right),
\]
where $x \in \Z^n$ is a lift of $\bar{x} \in \X_r$, $z \in \sX$ denotes a fixed arbitrary base point, and $\Delta(\gamma_0z , \dots , \gamma_{n-1}z)$ is defined as in the lemma above. 

\begin{proposition}\label{prop: representative of mu_r with R coefficients}
	We have $[\mu_{v_r^\ast {}_c E_\psi}] = \mu_r$ when viewed as classes in $H^{n-1}(\Gamma, \D(\X_r, \R))$.
\end{proposition}
\begin{proof}
	First observe that $\mu_{v_r^\ast {}_c E_\psi}$ is a group cocycle. This follows from the fact that ${}_c E_\psi$ is closed and invariant under $\Gamma$. Now, the proposition follows from observing that $[\mu_{v_r^\ast {}_c E_\psi}]$ maps to $[c_{v_r^\ast {}_c E_\psi}]$ via the isomorphism given by Shapiro's lemma 
	\[
	H^{n-1}(\Gamma, \D(\X_r, \R)) \xlongrightarrow{\sim} H^{n-1}(\Gamma_r, \R)
	\]
    described in \eqref{eq: iso from Shapiro's lemma}, Lemma \ref{lemma: representative of c_r over R}, and the definition of $\mu_r$ (see Definition \ref{def: mu_r}).
\end{proof}

Consider the $\Gamma$-equivariant morphism given by taking the total mass of a distribution
\[
\D(\X_1, \R) \too \R, \ \lambda \mapstoo \sum_{x \in \X_1} \lambda(x).
\]

\begin{corollary}\label{cor: mu1 has total mass zero}
	The corestriction map $H^{n-1}(\Gamma_1, \R) \to H^{n-1}(\Gamma, \R)$ maps $c_1$ to $0$. In particular, the morphism induced by taking the total mass of a measure 
    \[
    H^{n-1}(\Gamma_1, \D(\X_1, \R)) \too H^{n-1}(\Gamma, \R)
    \]
    maps $\mu_1$ to $0$.   
\end{corollary}
\begin{proof}
	The corestriction map can be written as 
	\[
	H^{n-1}(\Gamma_1, \R) \xlongrightarrow{\sim} H^{n-1}(\Gamma, \D(\X_1 , \R) ) \too H^{n-1}(\Gamma, \R),
	\]
	where the first map is given by the inverse of the map given by Shapiro's lemma, and the second one is the map induced by taking the total mass of a measure (see \cite[Chapter III, Section 9 (A)]{browncohomology}). In view of this observation and of Proposition \ref{prop: representative of mu_r with R coefficients}, it is enough to prove that the image of $[\mu_{v_1^\ast {}_c E_\psi}]$ by the second map is trivial. For that, observe that such image is represented by the cocycle
	\[
	(\gamma_0, \dots, \gamma_{n-1}) \mapstoo \int_{\Delta(\gamma_0z, \dots, \gamma_{n-1}z)}  \sum_{\bar{x} \in \X_1} (x/p)^\ast {}_c E_\psi.
	\]
	It follows from Proposition \ref{prop: average of pullbacks cEpsi by vectors of exact order p is 0} that the sum of differential forms in the integral is equal to zero, giving the desired result. 	
\end{proof}

\subsection{Lifting to measures of total mass zero}\label{subsec: lfting to measures of total mass zero} To construct rigid classes, it is useful to lift the class $\mu$ to a class with coefficients in measures of total mass zero. Let $\D_0 \coloneq \D_0(\X, \Z[1/m])$ be the sub-module of $\D \coloneq  \D(\X, \Z[1/m])$ consisting of measures $\lambda \in \D$ such that $\lambda(\X) = 0$. Consider the short exact sequence
\begin{equation}\label{eq: ses total mass}
0 \too \D_0 \too \D \too \Z[1/m] \too 0.
\end{equation}

\begin{proposition}\label{prop: image of mu is torsion} 
	The image of $\mu$ by the map $H^{n-1}(\Gamma, \D) \to H^{n-1}(\Gamma, \Z[1/m])$ is torsion.
\end{proposition}
\begin{proof}
	The result follows from Proposition \ref{cor: mu1 has total mass zero}.
\end{proof}

As we explained above, $w = \mathrm{diag}(1, -1, 1, \dots, 1) \in \GL_n(\Z)$ acts on the cohomology groups $H^i(\Gamma, \Z[1/m])$, $H^i(\Gamma, \D_0)$, and $H^i(\Gamma, \D)$. Moreover, \eqref{eq: ses total mass} yields a long exact sequence
\[
0 = H^{n-2}(\Gamma, \Q)^{-} \too H^{n-1}(\Gamma, \D_0)^{-}_{\Q} \too H^{n-1}(\Gamma, \D)^{-}_{\Q} \too H^{n-1}(\Gamma, \Q)^{-},
\]
where the subindex denotes taking the tensor product with $\Q$ over $\Z[1/m]$ and we used Lemma \ref{lemma: H^i(Gamma_r , Z[1/m])^- finite for i <= n-2} for the vanishing of $H^{n-2}(\Gamma, \Q)^{-}$. Thus, by Proposition \ref{prop: image of mu is torsion}, $\mu$ admits a unique lift to a class in $H^{n-1}(\Gamma, \D_0)^{-}_{\Q}$. 

\begin{definition}
    Let \begin{equation}\label{eq: mu0 fixed}
    \mu_0 \in H^{n-1}(\Gamma, \D_0(\X, \Z[1/m]))^{-}_{\Q}
    \end{equation}
    be a a lift of $\mu \in H^{n-1}(\Gamma, \D)_{\Q}$.
\end{definition}

\begin{remark} \label{rem:dependc3}
   If $a$ is any integer prime to $p$, we denote $[a]_*$ as the $\GL_n(\Z)$-equivariant operator on $\D(\X_r)$ or $\D(\X)$ given by pushforward of measures along the multiplication-by-$a$ map $\X_r \to \X_r$ (resp. $\X \to \X$). Then Remark \ref{rem:dependc} implies that for any $c$ and $d$ coprime to $p$ and each other 
    \[
        ([c]_*^{-1}-c^n){}_d\mu_r = ([d]_*^{-1}-d^n){}_c\mu_r,
    \]
    where the pre-subscripts, as before, denote the class associated to the corresponding smoothings. Then Proposition \ref{prop: class valued on measures in terms of inverse limit of classes} implies the same for the inverse limit class $\mu$. Note that the pullback $[a]^*$ on the cohomology of the torus induces $[a]_*^{-1}$ on the distributions over torsion specializations. From there, we deduce that, up to torsion, 
\[
        ([c]_*^{-1}-c^n){}_d\mu_0 = ([d]_*^{-1}-d^n){}_c\mu_0.
\]
\end{remark}

\section{Drinfeld's symmetric domain and log-rigid classes}\label{sec: Drinfeld's $p$-adic symmetric domain and rigid cocycles}

In this section, we introduce Drinfeld's $p$-adic symmetric domain $\sX_p$. Then, we define a lift from measures on $\X = \Z_p^n - p\Z_p^n$ of total mass zero to log-rigid analytic functions on $\sX_p$. This leads to construct a log-rigid class $J_{E, \mathcal L}$ as the image of the class $\mu_0 \in H^{n-1}(\Gamma, \D_0(\X, \Z[1/m]))_{\Q}$ of the previous section by such lift. We conclude by defining the evaluation of $J_{E, \mathcal L}$ at points $\tau \in \sX_p$ attached to totally real fields where $p$ is inert.

\subsection{Drinfeld's domain and rigid functions}
Drinfeld's $p$-adic symmetric domain is defined as $\sX_p \coloneq \mathbb{P}^{n-1}(\C_p) - \bigcup_{H \in \mathcal H} H$, where $\mathcal H$ is the set of all $\Q_p$-rational hyperplanes. It has the structure of a rigid analytic space, which we proceed to describe following \cite{SchneiderStuhler}.

For a given $H \in \mathcal H$, let $\ell_H$ be an equation of $H$ such that its coefficients form a unimodular vector in $\C_p^n$. Also, if $z \in \mathbb{P}^{n-1}(\C_p)$, we will always assume $z = [(z_0, \dots, z_{n-1})]$ is represented by a vector with unimodular coordinates. For $m \geq 1$, define
\[
	\sX_p^{\leq m} \coloneq \{ z \in \mathbb{P}^{n-1}(\C_p) \mid \ord_p( \ell_H(z) ) \leq m, \text{ for all $H \in \mathcal H$}  \}.
\]
The family $\{\sX_p^{\leq m}\}_m$ forms an admissible covering of $\sX_p$ by open affinoid subdomains. 

The ring of rigid functions on $\sX_p^{\leq m}$ can be described as follows. Let $\mathcal H_m$ be the set of equivalence classes of $\mathcal H$ modulo $p^m$. Also, fix $\bar{\mathcal H}_{m}$ a set of representatives for the equivalence classes in $\mathcal H_{m+1}$ containing the coordinate hyperplanes $H_i = \{ z_i = 0 \}$ for every $i = 0, \dots, n-1$. For $H, H' \in \mathcal H$, define the function $f_{H,H'}\colon \sX_p \to \C_p$
\[
f_{H,H'}(z)\coloneq \frac{\ell_H(z)}{\ell_{H'}(z)}.
\]
Then, observe that we can describe 
\[
\sX_p^{\leq m} = \{z \in \sX_p \mid \ord_p( f_{H, H'}(z) ) \geq -m \text{ for all $H, H' \in \bar{\mathcal H}_m$} \}.
\]
Let $A_m$ be the affinoid $\Q_p$-algebra obtained as the quotient of the free Tate algebra over $\Q_p$ in the indeterminates  $\{ T_{H, H'} \}_{H, H' \in \bar{\mathcal H}_m}$
modulo the closed ideal generated by 
\[
\begin{split}
	& T_{H,H} - p^m,  \text{ for $H \in \bar{\mathcal H}_m$} \\
	& T_{H, H'} T_{H', H''} - p^m T_{H, H''}, \text{ for $H, H', H'' \in \bar{\mathcal H}_m$}, \\
	& T_{H, H_j} - \sum_{i = 0}^{r-1} \lambda_i T_{H_i, H_j}, \text{ if $\ell_H(z) = \sum_{i = 0}^{n-1}\lambda_i z_i$ for $H \in \bar{\mathcal H}_m$ and $0 \leq j \leq n-1$}.
\end{split}
\]
The previous descriptions of $\sX_p^{\leq m}$ and $A_m$ lead to the following result.

\begin{proposition}
	Denote by $\mathcal A^{\leq m}$ the ring of rigid analytic functions on $\sX_p^{\leq m}$. Then, we have an isomorphism of $\Q_p$-algebras
	\[
	A_m \xlongrightarrow{\sim} \mathcal A^{\leq m}, \ T_{H, H'} \mapstoo  p^m f_{H, H'}.
	\]
	In particular, it induces an isomorphism of rigid spaces $\sX_p^{\leq m} \xlongrightarrow{\sim} \mathrm{Sp}(A_m)$.
\end{proposition}
\begin{proof}
	See proof of Proposition 4 of \cite{SchneiderStuhler}.
\end{proof}

In particular, $\mathcal A^{\leq m}$ is a Banach algebra with respect to the supremum norm.

\begin{definition}
	The ring of rigid analytic functions on $\sX_p$, denoted by $\mathcal A$, is the space of functions $f\colon \sX_p \to \C_p$ such that for every $m$, their restriction to $\sX_p^{\leq m}$ belongs to $\mathcal A^{\leq m}$. 
\end{definition}

We will also consider a larger space of functions on $\sX_p$, called log-rigid analytic functions. Let $\log_p\colon \C_p^\times \to \C_p$ be the branch of the $p$-adic logarithm satisfying $\log_p(p) = 0$. A function $f\colon \sX_p^{\leq m} \to \C_p$ is log-rigid analytic on $\sX_p^{\leq m}$ if it can be written as 
\[
f = f_0 + \sum_{H, H' \in \mathcal H} c_{H, H'}\log_p(f_{H, H'}(z)),
\]
where $f_0 \in \mathcal A^{\leq m}$ and $c_{H, H'} \in \Q_p$ are all but finitely many equal to $0$. Denote the space of log-rigid analytic functions on $\sX_p^{\leq m}$ by $\mathcal A^{\leq m}_{\mathcal L}$. 

\begin{definition}
	The space of log-rigid analytic functions on $\sX_p$, denoted by $\mathcal A_{\mathcal L}$, is the space of functions $f\colon \sX_p \to \C_p$ such that for every $m$, their restriction to $\sX_p^{\leq m}$ belongs to $\mathcal A^{\leq m}_{\mathcal L}$. 
\end{definition}

The following lemma will be useful to study log-rigid functions in the next sections.

\begin{lemma}\label{lemma: log(l_H/l_H') is rigid analytic if H and H' congruent}
    Let $m \geq 1$, and let $H, H' \in \mathcal H$ be hyperplanes with equations $\ell_H$ and $\ell_{H'}$ which are congruent modulo $p^{m+1}$. Then, the function 
    \[
    f\colon \sX_p^{\leq m} \too \C_p, \ z \mapstoo \log_p\left( f_{H,H'}(z) \right)
    \]
    is rigid analytic on $\sX_p^{\leq m}$.
\end{lemma}
\begin{proof}
    Observe that we can write
    \[
    f(z) = \log_p\left( 1 - \frac{ \ell_{H'}(z) - \ell_H(z)  }{\ell_{H'}(z)} \right).
    \]
    Moreover, since $\ell_H \equiv \ell_{H'} \mod p^{m+1}$ and $z \in \sX_p^{\leq m}$, we have 
    \[
    \ord_p\left( \frac{ \ell_{H'}(z) - \ell_H(z)  }{\ell_{H'}(z)}  \right) \geq 1.
    \]
    Therefore, 
    \[
    f(z) = \sum_{k \geq 1} \frac{1}{k} \left( \frac{ \ell_{H'}(z) - \ell_H(z)  }{\ell_{H'}(z)}   \right)^k,
    \]
    which is rigid analytic on $\sX_p^{\leq m}$. 
\end{proof}

Observe that matrix multiplication induces a right action of $\SL_n(\Q_p)$ on $\sX_p$ given as follows. For $g \in \SL_n(\Q_p)$ and $z \in \sX_p$ represented by a vector in $\C_p^n$, that we also denote by $z$, we have
\[
(z,g) \coloneq [g^t z],
\]
where $g^t \in \SL_n(\Q_p)$ denotes the transpose of $g$. This induces a left action of $\SL_n(\Q_p)$ on the space of $\C_p$-valued functions on $\sX_p$.  If $g \in \SL_n(\Q_p)$, $f$ is a function on $\sX_p$, and $z \in \sX_p$  
\[
(g \cdot f)(z) \coloneq f(g^{t}z).
\]
This action preserves the subspaces $\mathcal A$ and $\mathcal A_{\mathcal L}$.

\subsection{Lifts from measures to functions on $\sX_p$} Recall that $\sX_p$ consists of the points in $\mathbb{P}^{n-1}(\C_p)$ that do not belong to a $\Q_p$-rational hyperplane. On the other hand, a point in $\X = \Z_p^n - p\Z_p^n$ gives the equation of a $\Q_p$-rational hyperplane. This suggests considering the two-variable function  
\[
\left( \C_p^n - \bigcup_{H \in \mathcal H} H \right) \times \X \too \C_p, \ (z , x) \mapstoo \log_p( z^t \cdot x), 
\]
Integration with respect to the variable $x \in \X$ will induce a map from total mass zero measures on $\X$ to functions on $\sX_p$.

\begin{lemma}
    Let $\lambda \in \D_0(\X, \Z_p)$. The function $F\colon \sX_p \too \C_p$ given by
    \[
    z \mapstoo F(z) \coloneq \int_{\X} \log_p(z^t \cdot x) d\lambda,
    \]
    where $z$ in the right hand side denotes an arbitrary representative in $\C_p^n$ of $z \in \sX_p$,
    is well-defined and belongs to $\mathcal A_{\mathcal L}$.
\end{lemma}
\begin{proof}
    For every $r \geq 1$, fix $V_r$ a set of representatives in $\Z^n$ of $\X_r = (\Z/p^r\Z)^n - (p\Z/p^r\Z)^n $ and define 
    \[
    f_r\colon \sX_p \too \C_p, \ z \mapstoo \sum_{v \in V_r} \lambda(U_{v/p^r}) \log_p(z^t \cdot v),
    \]
    where $U_{v/p^r} \subset \X$ is as in \eqref{eq: def of U_x/p^r}. Observe that since $\lambda(\X) = 0$, $f_r(z)$ is independent of the choice of representative of $z$ in $\C_p^n$, showing that $f_r$ is a well-defined function. For the rest of the proof we will assume that the representative of $z$ (also denoted $z$) is chosen so that its coordinates are unimodular. We follow the next steps:
    \begin{itemize}
        \item $F$ is a well-defined function on $\sX_p$. Indeed, for $z \in \C_p^n - \bigcup_{H \in \mathcal H} H$, the function $x \in \X \mapsto \log_p(z^t \cdot x)$ is continuous on the compact set $\X$. Thus, the integral defining $F(z)$ converges and we have pointwise convergence
        \[
        F(z) = \lim_{r \to +\infty} f_r(z).
        \]

        \item The sequence $({f_r}_{\vert \sX_p^{\leq m}})$ converges to $F_{\vert \sX_p^{\leq m}}$ with respect to the sup norm for $m \geq 1$. To simplify the notation, denote by $(f_r)$ and $F$ the restrictions of these functions to $\sX_p^{\leq m}$. To prove that $(f_r)_r$ converges to $F$ with respect to the sup norm it is enough to see that $(f_r)_r$ is Cauchy with respect to this norm. Observe that, if we let 
        $\pi\colon V_{r+1} \twoheadrightarrow{} V_r$ be the lift of the reduction modulo $p^r$ map $\X_{r+1} \twoheadrightarrow{} \X_r$ and use that $\lambda$ is a measure, we have
        \[
        \begin{split}
        f_{r+1}(z) - f_r(z) & = \sum_{v \in V_{r+1}} \lambda(U_{v/p^{r+1}} ) \log_p \left( \frac{z^t \cdot v}{ z^t \cdot \pi(v)} \right) \\ & = \sum_{v \in V_{r+1}} \lambda(U_{v/p^{r+1}}) \log_p\left(  1 + \frac{z^t \cdot (v - \pi(v))}{z^t \cdot \pi(v)}  \right).
        \end{split}
        \]
        Since $v \equiv \pi(v) \mod p^r$, we deduce that for every $z \in \sX_p^{\leq m}$ 
        \[
        \ord_p\left( \frac{z^t \cdot (v - \pi(v))}{z^t \cdot \pi(v)} \right) \geq r - m,
        \]
        Thus, if $r > m$, we can use the power series expansion of $\log(1 + x)$ to deduce that 
        \[
        \ord_p(f_{r+1}(z) - f_r(z)) \geq r-m \ \text{for all $z \in \sX_p^{\leq m}$}
        \]
        It follows from there that $(f_r)_r$ is Cauchy.

        \item $F \in \mathcal A_{\mathcal L}$. Let $m \geq 1$ and denote by $(f_r)_r$ and $F$ the restrictions of these functions to $\sX_p^{\leq m}$. It is enough to see that  $F$ belongs to $\mathcal A_{\mathcal L}^{\leq m}$. With this aim, write
        \[
        F = \left( \lim_{r \to +\infty} (f_r - f_{m+1}) \right)+ f_{m+1}.
        \]
        We claim that $\lim_{r \to +\infty} (f_r - f_{m+1})$ is a rigid analytic function. Indeed, we can write
        \[
        f_r(z) - f_{m+1}(z) = \sum_{v \in V_{r+1}}\lambda(U_{v/p^{r+1}}) \log_p\left( \frac{z^t \cdot v}{ z^t \cdot \pi^{r-(m+1)}(v)}  \right).
        \]
        Since $v \equiv \pi^{r- (m+1)}(v) \mod p^{m+1}$, it follows from Lemma \ref{lemma: log(l_H/l_H') is rigid analytic if H and H' congruent}, that $f_r - f_{m+1}$ is rigid analytic on $\sX_p^{\leq m}$. Then, since the sequence $(f_r - f_{m+1})_r$ converges with respect to the sup norm by the previous point of this proof, and $\mathcal A^{\leq m}$ is complete with respect to this norm, we deduce the desired claim.

        On the other hand, since $\lambda$ has total mass zero, we have that $f_{m+1} \in A_{\mathcal L}^{\leq m}$, as it can be written as a linear combination of $\log_p(f_{H, H'}(z))$ for $\Q_p$-rational hyperplanes $H, H' \in \mathcal H$. Hence, we deduce that $F \in \mathcal A_{\mathcal L}^{\leq m}$ and we are done.
    \end{itemize}

    \end{proof}

In view of the previous lemma, we can define a lift from measures of total mass zero to log-rigid analytic functions on $\sX_p$.

\begin{definition}
    Let $\mathrm{ST}$ be the morphism given by
    \[
    \mathrm{ST} \colon \D_0(\X, \Z[1/m]) \too \mathcal A_{\mathcal L} , \ \lambda \mapstoo \left( z \mapstoo \int_\X \log_p(z^t \cdot x)d\lambda \right).
    \]
\end{definition}
The morphism $\mathrm{ST}$ is $\Gamma$-equivariant. Therefore, it induces a map in cohomology 
\[
\mathrm{ST} \colon H^{n-1}(\Gamma, \D_0(\X, \Z[1/m])) \too H^{n-1}(\Gamma, \mathcal A_{\mathcal L}).
\]
Using this map, we obtain our desired log-rigid analytic class. 

\begin{definition}
    Let $\mu_0 \in H^{n-1}(\Gamma, \D_0(\X, \Z[1/m]))_{\Q}$ be as in \eqref{eq: mu0 fixed}. Define 
    \[
    J_{E, \mathcal L} \coloneq \mathrm{ST}(\mu_0) \in H^{n-1}(\Gamma, \mathcal A_{\mathcal L})_{\Q}.
    \]
\end{definition}

As with the cocycles $\mu_0$, we have the following independence-of-$c$ result, where we here, as before, denote dependence on the smoothing with a pre-subscript.
\begin{proposition}
    If $c$, $d$ are coprime to each other and also to $p$, we have 
    \[
        (1-d^n){}_cJ_{E, \mathcal L}=(1-c^n){}_dJ_{E, \mathcal L}.
    \]
\end{proposition}
\begin{proof}
    {For any prime-to-$p$ scalar $a$, we have $\mathrm{ST} \circ [a]_* = \mathrm{ST}$, as if $\lambda \in \D_0(\X, \Z[1/m])$
    \[
    \int_{\X} \log_p(z^t\cdot x)\,d([a]_*\lambda) = \int_{\X}\log_p a +\log_p(z^t\cdot x)\,d\lambda = \int_{\X}\log_p(z^t\cdot x)\,d\lambda
    \]
    with the last equality by $\lambda(\X)=0$. Then the result follows by passing to group cohomology for $\Gamma$ and applying Remark \ref{rem:dependc3}.}
\end{proof}

In particular, $(1-c^n)^{-1}{}_cJ_{E, \mathcal L}$ is independent of $c$, though this introduces denominators.

\begin{remark}\label{remark: comparison J_Eis with JDR when n = 2}
Suppose $n = 2$ and consider $\mathcal{J}_{\mathrm{DR}} \in H^1(\SL_2(\Z), \mathcal A^\times)^{-}$ a lift of (the restriction to $\SL_2(\Z)$ of) $J_{\mathrm{DR}} \in H^1(\SL_2(\Z[1/p]), \mathcal A^\times/\C_p^\times)^-$ constructed in \cite{DPV2}. By comparing the constructions of $J_{\mathrm{DR}}$ and $J_{E, \mathcal L}$, we deduce $J_{E, \mathcal L} = \log_p(\mathcal {J}_{\mathrm{DR}})$.  
\end{remark}

\subsection{Evaluation at totally real fields where $p$ is inert}

Let $F$ be a totally real field of degree $n$ where $p$ is inert and denote by $\sigma_1, \dots, \sigma_n$ the collection of embeddings of $F$ into $\R$. Let $\mathfrak a$ be an integral ideal of $F$ of norm coprime to $pc$. Fix $\{ \tau_1, \dots, \tau_n \}$ an oriented $\Z$-basis of $\mathfrak a^{-1}$, in the sense that the square matrix $(\sigma_i(\tau_j))_{i,j}$ has positive determinant, and let $\tau \in F^n$ be the column vector whose $i$th entry is equal to $\tau_i$. The vector $\tau$ induces an isomorphism of $\Q$-vector spaces 
\[
\Q^n \xlongrightarrow{\sim} F, \ x \mapstoo \tau^t \cdot x. 
\]
The action of multiplication by $F^\times$ on $F$, which is $\Q$-linear, gives an embedding
\begin{equation}\label{eq: F inside Mn(Q)}
F \intoo \mathrm{M}_n(\Q), \ \alpha \mapstoo A_\alpha
\end{equation}
determined by the following property: for $\alpha \in F$ and $x \in \Q^n$, $\alpha (\tau^t \cdot x) = \tau^t \cdot ( A_{\alpha} x)$.

\begin{lemma}
	The element $\tau \in \mathbb{P}^{n-1}(\C_p)$ belongs to $\sX_p$ and is fixed by $F^1 \xhookrightarrow{} \SL_n(\Q)$.
\end{lemma}
\begin{proof}
	The coordinates of $\tau$ give a $\Q$-basis of $F$. Since $p$ is inert in $F$, the coordinates of $\tau$ also form a $\Q_p$-basis of the completion of $F$ at $p$. In particular, they are independent over $\Q_p$. In other words, $\tau \in \sX_p$. Finally, for every $\alpha \in F$ we have $A_\alpha^t \tau = \alpha \tau$ by the property stated below \eqref{eq: F inside Mn(Q)}. In particular, $\tau \in \sX_p$ is fixed by the action of $F^1 \xhookrightarrow{} \SL_n(\Q)$.
\end{proof}

Let $U_F$ be the subgroup of totally positive units in $\cO_F^\times$. We view $U_F$ as a subgroup of $\Gamma$. Consider the following morphism in cohomology induced by evaluation at $\tau$
\[
H^{n-1}(\Gamma, \mathcal A_{\mathcal L}) \xlongrightarrow{\mathrm{ev}_{\tau}} H^{n-1}(U_F,  \C_p ).
\]
By Dirichlet's unit theorem, $U_F \simeq \Z^{n-1}$.  Therefore, $H_{n-1}(U_F, \Z) \simeq \Z$, and we can fix a generator of this group $c_{U_F} \in H_{n-1}(U_F, \Z)$.

\begin{definition}
	Consider the same notation as above, and let $J \in H^{n-1}(\Gamma, {\mathcal A}_{\mathcal L})_{\Q}$. Define the evaluation of $J$ at $[\tau] \in \sX_p$ by the cap product
	\[
	J[\tau] \coloneq c_{U_F} \frown \mathrm{ev}_{\tau}(J) \in \C_p.
	\]
\end{definition}

Since $J_{E, \mathcal L} = \mathrm{ST}(\mu_0)$, it follows from the description of the map $\mathrm{ST}$ that $J_{E, \mathcal L}[\tau] \in F_p$.
We also note that this definition depends, up to a sign, of the choice of generator $c_{U_F} \in H_{n-1}(U_F, \Z)$. In the next section, we will make a precise choice of generator when comparing the local trace of these values to the local trace of $p$-adic logarithms of Gross--Stark units.

\section{Traces of values of the log-rigid class and the Gross--Stark Conjecture}\label{sec: Values of the log-rigid class and Gross--Stark conjecture}

Let $F$ be a totally real field where $p$ is inert, let $\mathfrak a$ be an integral ideal of $F$ coprime to $pc$, and fix $\tau \in F^n$ a vector whose entries give an oriented $\Z$-basis of $\mathfrak a^{-1}$, which yields a point $\tau \in X_p$. Recall the log-rigid analytic class $J_{E, \mathcal L}$ constructed in the previous section and the value $J_{E, \mathcal L}[\tau] \in F_p$. In this section, we prove 
\[
\mathrm{Tr}_{F_p/\Q_p} J_{E, \mathcal L}[\tau] = -L_p'(1_{[\mathfrak a],p},0),
\]
where $L_p(1_{[\mathfrak a],p},s)$ denotes a $p$-adic partial zeta function attached to the class of $\mathfrak a$ in the narrow Hilbert class group of $F$. From this expression and the rank $1$ Gross--Stark conjecture, we obtain the equality
\[
\mathrm{Tr}_{F_p/\Q_p} J_{E, \mathcal L}[\tau] = \mathrm{Tr}_{F_p/\Q_p} \log_p(u^{\sigma_{\mathfrak a}})
\] 
for $u^{\sigma_\mathfrak a}$ a Gross-Stark unit in the narrow Hilbert class field of $F$ attached to the class of $\mathfrak{a}$. 

\subsection{$p$-adic $L$-functions and Gross--Stark conjecture} We state the the Gross--Stark conjecture in a simple setting. For more details, we refer the reader to \cite[Section 3]{Gross1981padicL} and \cite[Section 2]{DasguptaShintani}. We begin by introducing the following notation. For an integral ideal $\mathfrak f$ of $F$, denote by $G_{\mathfrak f}$ the narrow ray class group modulo $\mathfrak f$. It is obtained by taking the quotient of the set of integral ideals in $F$ which are prime to $\mathfrak f$ by the relation 
\[
\mathfrak b \sim_{\mathfrak f} \mathfrak c \text{ if and only if } \mathfrak b \mathfrak c^{-1} = (\lambda) \text{ for $\lambda \in 1 + \mathfrak f \mathfrak c^{-1}$ totally positive.}
\]
Then, if $\varepsilon$ is a $\bar{\Q}$-valued function on $G_{\mathfrak f}$, we let
\[
L(\varepsilon , s) \coloneq \sum_{(\mathfrak b, \mathfrak f) = 1 } \varepsilon(\mathfrak b) \mathrm{N} \mathfrak b^{-s},
\]
where the sum is over integral ideals which are coprime to $\mathfrak f$. This sum converges for $s \in \C$ such that $\mathrm{Re}(s) > 1$ and it can be extended via analytic continuation to a meromorphic function at $\C$ with at most a pole at $s = 1$, that we will still denote by $L(\varepsilon, s)$. Recall that $c$ is a positive integer prime to $p$ and denote by $\varepsilon_c$ the function on $G_{\mathfrak f}$ given by $\varepsilon_c(\mathfrak b) = \varepsilon((c)\mathfrak b)$. For $k \in \Z_{\geq 1}$, consider
\[
\Delta_{c}(\varepsilon, 1 - k) \coloneq L(\varepsilon, 1 - k ) - c^{nk}  L( \varepsilon_c, 1 - k ).
\]
It is result of Klingen and Siegel that $\Delta_{c}(\varepsilon, 1- k) \in \Q(\varepsilon)$, where $\Q(\varepsilon)$ denotes the field generated by the values of $\varepsilon$. Deligne--Ribet and Cassou--Nogu\`{e}s refined this statement by studying the integrality properties of these values. Their study results in the existence of $p$-adic analytic functions interpolating these values, which we proceed to outline for the case of partial zeta functions. 

Let $G \coloneq \varprojlim_{r \geq 1} G_{p^r}$, where the limit is taken with respect to the natural projection maps $G_{p^{r+1}} \to G_{p^r}$, let $H_{\mathfrak a }$ be the open subset of $G$ consisting of the pre-image of $\mathfrak a$ via the natural map $G \to G_1$, and denote by $1_{[\mathfrak a], p} \colon G \to \Z$ the characteristic function of $H_{\mathfrak a}$. If $\varepsilon \colon G \to \Z$ is locally constant, it factors through $G_{p^r}$ for some $r \geq 1$. We then define $L(\varepsilon ,s)$ by viewing $\varepsilon$ as a function on $G_{p^r}$, which is independent of the choice of $r$.
\begin{theorem}\label{thm: Deligne--Ribet p-adic L-function}
    For $\varepsilon \colon H_{\mathfrak a} \to \Z$ locally constant, consider the product $\varepsilon 1_{[\mathfrak a], p}$ and view it as a locally constant function on $G$.
    \begin{enumerate}
        \item If $k \geq 1$, we have $\Delta_c(\varepsilon 1_{[\mathfrak a], p}, 1-k) \in \Z[1/c]$.
        \item The distribution $\mu_{\mathfrak a} \colon\varepsilon \mapsto \Delta_c(\varepsilon1_{[\mathfrak a], p}, 0)$
        defines a measure on $H_{\mathfrak a}$.
        \item The function 
        \[
        L_p(1_{ [\mathfrak a], p } , \cdot) \colon \Z_p \too \Z_p, \ s \mapstoo \int_{H_{\mathfrak a}} \langle \Norm \mathfrak b\rangle^{-s}  d\mu_{\mathfrak a}(\mathfrak b)
        \]
        is analytic and is characterized by the following interpolation property:  for every integer $k \geq 1$ such that $k \equiv 1 \mod [F(\mu_{2p}) : F]$, 
        \begin{equation}\label{eq: interpolation property of L_p at 1- k}
        L_p(1_{ [\mathfrak a], p } , 1 - k) = \Delta_c(1_{[\mathfrak a], p} , 1 - k ).
        \end{equation}
            \end{enumerate}
\end{theorem}
\begin{proof}
    This follows from Theorem 0.5 of \cite{DR}.
\end{proof}
Observe that $L(1_{[\mathfrak a],p}, s)$ is a partial zeta function with the Euler factor corresponding to $p$ removed. This implies that $\Delta_c(1_{[\mathfrak a], p} , 0 ) = 0$ and, by \eqref{eq: interpolation property of L_p at 1- k}, $L_p(1_{ [\mathfrak a], p } , 0) = 0$ as well. The Gross--Stark conjecture gives an arithmetic interpretation for the value of the derivative $L_p'(1_{ [\mathfrak a], p }, 0)$ with respect to $s$ at $s = 0$. For that, let $H$ be the narrow Hilbert class field of $F$ and consider the following subgroup of $p$-units in $H$
\[
\cO_H[1/p]_-^\times \coloneq \{ x \in H^\times \mid \lvert x \rvert_{\mathfrak q} = 1 \ \forall \mathfrak q \nmid p  \},
\]
where $\mathfrak q$ runs over all archimedean and nonarchimedean places of $H$ not dividing $p$. The embedding $\bar{\Q} \subset \bar{\Q}_p$ determines a prime $\mathfrak{p}$ of $H$ above $p$ that we fix here and from now on. The following proposition is due to Gross, see \cite[Proposition 3.8]{Gross1981padicL}.
\begin{proposition}\label{prop: determination of u}
	There exists a unique element $u \in \cO_H[1/p]_-^\times \otimes \Q$ satisfying 
	\[
	\ord_{\mathfrak p}(u^{\sigma_{\mathfrak a}}) = \Delta_c(1_{[\mathfrak a]} , 0) \text{ for all $\mathfrak {a}$ coprime to $p$,} 
	\]
	where $1_{[\mathfrak a]}$ denotes the characteristic function of $[\mathfrak a]$ on $G_1$ and, here and from now on, $\sigma_{\mathfrak{a}} \in \Gal(H/F)$ denotes the Frobenius element associated to $\mathfrak a$.
\end{proposition}

\begin{remark}\label{rmk: from Gross--Stark to Brumer--Stark}
    If every prime factor of $c$ is greater than $n + 1$, the Brumer--Stark conjecture, proven in \cite{DKBrumerStark} and \cite{DKSW}, implies that in fact $u \in \cO_H[1/p]^\times_-$. Indeed, the quantities $\Delta_c(\varepsilon, 1 - k)$ can be written as a linear combination of values of (smoothed) partial zeta functions $\zeta_{S,T}(\sigma, 1- k)$ considered in \cite{DasguptaShintani} for $T$ running over subsets of the set of prime ideals of $F$ dividing $c\cO_F$. Under the condition on $c$ given above, each of these subsets satisfies the assumptions to apply the proof of Brumer--Stark, see \cite[Section 1.1]{DKSW}. 
\end{remark}

Since $p\cO_F$ splits completely on $H$, we have $H \subset H_{\mathfrak p} = F_p$. 

\begin{theorem}[Gross--Stark conjecture]\label{thm: Gross--Stark conj}
	Let $u$ be as in Proposition \ref{prop: determination of u}. We have
	\[
	L_p'(1_{[\mathfrak a] , p} , 0) = - \log_p (\Norm_{F_p/\Q_p}u^{\sigma_{\mathfrak a}}) \text{ for all $\mathfrak {a}$ coprime to $p$.} 
	\]
\end{theorem}
\begin{proof}
	See \cite{DDP} and \cite{Ventullo}.
\end{proof}

\subsection{Periods of the Eisenstein class along tori attached to totally real fields} We use the differential forms representing the Eisenstein class of Section \ref{sec: differential form representatives of Eisenstein class} to prove that pullbacks of the Eisenstein class by torsion sections encode special values of zeta functions of totally real fields. A general version of the result was proven in \cite[Section 12.6]{BCG} using an adelic framework, and we specialize their results and outline the proof below for the cases that will be relevant for us. Our calculations are similar to those in Section 4.2 of \cite{BergeronCharolloisGarcia2023}.

Recall that $F$ is a totally real field of degree $n$ where $p$ is inert, $\mathfrak a$ is an integral ideal of $F$ prime to $pc$, and $\tau \in F^n$ is a column vector whose entries give a positively oriented $\Z$-basis of $\mathfrak a^{-1}$. As we saw in the previous section, $\tau$ induces a $\Q$-linear isomorphism 
\begin{equation}\label{eq: iso beta Q^n = F}
\beta\colon \Q^n \xlongrightarrow{\sim} F, \ x \mapstoo \tau^t \cdot x. 
\end{equation}
The action of multiplication by $F^\times$ on $F$, which is $\Q$-linear, gives an embedding
\begin{equation}\label{eq: F inside Mn(Q)}
	F \intoo \mathrm{M}_n(\Q), \ \alpha \mapstoo A_\alpha
\end{equation}
determined by the following property: for all $\alpha \in F$ and $x \in \Q^n$, $\alpha (\tau^t \cdot x) = \tau^t \cdot ( A_{\alpha} x)$.
Let $(F\otimes \R)^1_+$ be the subset of totally positive elements of norm $1$. The embedding \eqref{eq: F inside Mn(Q)} induces an oriented map (see Section 12.4 of \cite{BCG} for more details on the orientation)
\[
i_{\tau}\colon (F \otimes \R)_+^1 \too \sX.
\]
Denote by $U_F$ the subgroup of totally positive units in $\cO_F^\times$. Since $U_F$ has rank $n-1$ by Dirichlet's unit theorem, it follows that
\begin{equation}\label{eq: X(F)}
X(F) \coloneq U_F \backslash (F\otimes \R)_+^1
\end{equation}
is a compact oriented manifold of dimension $n-1$. 

We now introduce a linear combination of pullbacks of the Eisenstein class that we will integrate along $X(F)$. For $r \geq 1$, let
\[
\chi \colon \left(\mathfrak a^{-1} - p\mathfrak a^{-1} \right) / p^r\mathfrak a^{-1} \too \bar{\Q}
\]
be an $\cO_F^\times$-invariant function, where here and from now on, $\left(\mathfrak a^{-1} - p\mathfrak a^{-1} \right) / p^r\mathfrak a^{-1}$ denotes the set $\mathfrak a^{-1} - p\mathfrak a^{-1}$ modulo the translation action by $p^r \mathfrak a^{-1}$. Recall $\X_r : = (\Z/p^r\Z)^n - (p\Z/ p^r \Z)^n$ and observe that dot product with $\tau$ induces a bijection 
\begin{equation}\label{eq: bijections Xr to a^{-1} - pa^{-1} mod p^r to (O/p^r)*}
\X_r \xlongrightarrow{\sim} \left(\mathfrak a^{-1} - p\mathfrak a^{-1} \right) / p^r\mathfrak a^{-1}.
\end{equation}
We will sometimes view $\chi$ as a function on $\mathfrak a^{-1}$ which is $0$ on $p\mathfrak a^{-1}$, and the rest of their values are determined by the value of $\chi$ on the residue classes modulo $p^r\mathfrak a^{-1}$. 
\begin{definition}
	Consider the same notation as above. Define 
	\begin{equation}\label{eq: definition of E_r, chi}
		E_{\tau, \chi} \coloneq \sum_{\bar{x} \in \X_r} \chi(c\tau^t \cdot x) \left( x/p^r \right)^\ast {}_c E_{\psi} \in \Omega^{n-1}(\sX),
	\end{equation}
	For $s \in \C$, we define $E_{\tau, \chi}(s)$ as above but replacing ${}_c E_\psi$ by ${}_c E_\psi(s) = E_\psi(c^{-1}\Z^n , s) - c^n E_\psi(\Z^n ,s)$ in the definition.
\end{definition}
	\begin{lemma}
		The differential form $E_{\tau,\chi}$ on $\sX$ is invariant under $U_F \subset \Gamma$, where the inclusion of $U_F$ in $\Gamma$ is induced by \eqref{eq: F inside Mn(Q)}.
	\end{lemma}
	\begin{proof}
		For $\gamma \in \Gamma$, note that we have $\gamma^\ast v^\ast {}_c E_\psi = (\gamma v)^\ast {}_c E_{\psi}$. Then, if $\gamma \in U_F \subset \Gamma$ 
		\[
		\begin{split}
			\gamma^\ast E_{\tau, \chi} & \coloneq   \sum_{\bar{x} \in \X_r} \chi(c\tau^t \cdot x) (\gamma x/p^r)^\ast {}_c E_\psi \\ & = \sum_{\bar{x} \in \X_r} \chi(c\tau^t  \cdot \gamma^{-1} x) (x/p^r)^\ast {}_c E_\psi \\ & = \sum_{\bar{x} \in \X_r} \chi(\varepsilon c\tau^t  \cdot x) (x/p^r)^\ast {}_c E_\psi  \\ & = \sum_{\bar{x} \in \X_r} \chi(c\tau^t \cdot x) (x/p^r)^\ast {}_c E_\psi,
		\end{split}
		\]
		where we used that $\tau^t \gamma^{-1} =   \varepsilon \tau^t$, for $\varepsilon \in U_F$ the preimage of $\gamma^{-1}$ by \eqref{eq: F inside Mn(Q)} and that $\chi$ is $U_F$-invariant.
	\end{proof}

Thus, $i_\tau^\ast E_{\tau, \chi}$ defines a closed form on $X(F)$ and we can consider 
\[
\int_{X(F)} i_\tau^\ast E_{\tau, \chi}.
\]
We will express this integral in terms of $L$-values. Observe that we have a bijection
\begin{equation}\label{eq: from a^-1/p^r a^{-1} to preimage of a in Gp^r}
\left( \left(\mathfrak a^{-1} - p\mathfrak a^{-1} \right) / p^r\mathfrak a^{-1} \right) / U_F \xlongrightarrow{\sim} \{\mathfrak b \in G_{p^r} \mid \mathfrak b \sim_1 \mathfrak a \} \intoo G_{p^r}, \ [\lambda] \mapstoo [\mathfrak a (\lambda)], 
\end{equation}
where $\lambda \in \mathfrak a^{-1}$ is a totally positive element in $[\lambda]$. We can use this bijection to consider 
\[
\chi \cdot 1_{[\mathfrak a],p} \colon G_{p^r} \xlongrightarrow{} \bar{\Q},
\]
where $1_{[\mathfrak a], p}$ denotes characteristic function of the preimage of $[\mathfrak a] \in G_1$ via the projection $G_{p^r} \to G_1$ and $\chi$ is viewed as a function on the preimage of $[\mathfrak a]$ in $G_{p^r}$ via the bijection above.  
\begin{lemma}\label{lemma: expression for L-function summing over a^{-1}}
    We have
	\[
	L(  \chi 1_{[\mathfrak a] , p}    , 0) = \lim_{s \to 0}\frac{1}{2^n} \sum_{ \alpha \in U_F \backslash \mathfrak a^{-1}  } \frac{ \chi( \alpha ) \mathrm{sign}(\Norm \alpha )  }{ \lvert \Norm \alpha \rvert ^{s}},
	\]
	where on the right hand side, $\lim_{s \to 0}$ denotes evaluation at $s = 0$ of the analytic continuation.
\end{lemma}
\begin{proof}
	The result can be deduced from equation (7.15) of \cite{HugoC}, which is originally due to Siegel (\cite{Siegel}).
\end{proof}

\begin{theorem}\label{thm: period of E_r,chi in terms of L-functions}
	Consider the same notation as above. Then,
	\[
	\int_{X(F)} i_{\tau}^\ast E_{\tau ,\chi} = \Delta_c\left( \chi 1_{[\mathfrak a] , p} ,0\right).
	\]
\end{theorem}
\begin{proof}
	For $s \in \C$ such that $\mathrm{Re}(s) \gg 0$, we have 
	\[
	\begin{split}
		& \int_{X(F)} i_{\tau}^\ast E_{r, \chi}(s) =  \int_{X(F)} i_{\tau}^\ast \left(  \sum_{\bar{v} \in \X_r} \chi(c\tau^t \cdot v) \left( v/p^r \right)^\ast {}_c E_{\psi}(s) \right) \\ & = \int_{X(F)} i_{\tau}^\ast \left( \sum_{\bar{v} \in \X_r} \chi(c\tau^t \cdot v)  \left( \sum_{\lambda \in v/p^r + c^{-1}\Z^n }  \eta(\lambda, s) - c^n  \sum_{\lambda \in v/p^r + \Z^n }  \eta(\lambda, s) \right)      \right),
	\end{split}
	\]
    where we recall that $\eta(\lambda, s)$ was introduced in Section \ref{subsec: pullbacks by torsion sections}. Using that $\eta(v/p^r, s) = p^s \eta(v, s)$, and keeping in mind that we will later be interested in evaluating the analytic continuation of the expression above at $s = 0$, it is enough to compute 
	\[
	\begin{split}
		& \int_{X(F)} i_{\tau}^\ast \left( \sum_{\bar{v} \in \X_r} \chi(c\tau^t \cdot v)  \left( \sum_{\lambda \in v + c^{-1}p^r\Z^n }  \eta(\lambda, s) - c^n  \sum_{\lambda \in v + p^r\Z^n }  \eta(\lambda, s) \right)      \right) \\ & = \int_{X(F)} i_{\tau}^\ast \left( \sum_{\bar{v} \in \X_r} \chi(c\tau^t \cdot v) \left( \sum_{x \in \beta(v) + c^{-1}p^r\mathfrak{a}^{-1} }  \eta(\beta^{-1}x, s) -  c^n  \sum_{x \in \beta(v) + p^r\mathfrak{a}^{-1} }  \eta(\beta^{-1}x, s) \right)      \right) \\ & = \int_{X(F)} i_{\tau}^\ast  \left(  \sum_{x \in c^{-1}  \mathfrak a^{-1}} \chi(cx) \eta(\beta^{-1}x, s) - c^n \sum_{x \in \mathfrak a^{-1}} \chi(cx) \eta(\beta^{-1} x, s) \right).		
	\end{split}
	\]
	We can compute the inner sums by first taking representatives of $U_F \backslash  c^{-1}\mathfrak a^{-1}$ and $U_F\backslash \mathfrak a^{-1}$, that we denote by $x$, and then running over all elements in $U_F$, denoted by $u$. Hence, we obtain that the previous expressions can be written as
	\[
	\begin{split}
		&\int_{X(F)} i_{\tau}^\ast \left( \sum_{U_F \backslash c^{-1} \mathfrak a^{-1}} \sum_{U_F} \chi(cux) \eta(\beta^{-1}ux, s)  - c^n \sum_{ U_F \backslash \mathfrak a^{-1}} \sum_{U_F} \chi(cux) \eta(\beta^{-1}ux, s)    \right) \\ & = \sum_{ U_F \backslash c^{-1}\mathfrak a^{-1} } \chi(cx)\int_{(F\otimes \R)^1_+} i_{\tau}^\ast \eta(\beta^{-1} x, s  ) - c^n \sum_{U_F \backslash \mathfrak a^{-1}} \chi(cx) \int_{(F \otimes \R)^1_+} i_{\tau}^\ast\eta(\beta^{-1}x, s).
	\end{split}
	\]
	From \cite[Section 12.8]{BCG}, we have that for $x \in F$
	\[
	\int_{(F\otimes \R)^1_+} i_{\tau}^\ast\eta(\beta^{-1}x, s) = \pi^{-n/2}2^{s/2 - n}\Gamma\left(\frac{s}{2n} + \frac{1}{2}\right)^n \frac{\mathrm{sign}(\Norm(x))}{\lvert \Norm(x)\rvert^s}.
	\] 
	Hence, we deduce 
	\[
	\int_{X(F)} i_\tau^\ast E_{r, \chi} = \frac{1}{2^n}\lim_{s \to 0} \sum_{x \in U_F \backslash \mathfrak a^{-1}   } \chi(x)\frac{\mathrm{sign}(\Norm(x))}{\lvert \Norm(x) \rvert^s} - c^n \sum_{x\in U_F \backslash \mathfrak a^{-1}} \chi(cx) \frac{\mathrm{sign}(\Norm(x))}{\lvert \Norm(x) \rvert^s}.
	\]
	Finally, the desired equality follows from Lemma \ref{lemma: expression for L-function summing over a^{-1}}.
\end{proof}

\subsection{The class $\mu$ and $p$-adic $L$-functions} We state the relation between the class $\mu$ constructed in Section \ref{sec: Topological construction of the Eisenstein group cohomology class} and the $p$-adic $L$-function $L_p(1_{[\mathfrak a],p},s)$ introduced above. From there, we relate $\mathrm{Tr}_{F_p/\Q_p}J_{E, \mathcal L}[\tau]$ to traces of $p$-adic logarithms of Gross--Stark units. 

Denote by $\mathfrak a_p$ the completion of $\mathfrak a $ at $p$. For $\chi\colon \mathfrak a^{-1}_p - p \mathfrak a^{-1}_p \to \bar{\Q}_p$ a continuous function that is $\cO_F^\times$-equivariant, define the map
\[
\varphi_\chi\colon \D(\X, \Z[1/m]) \too \bar{\Q}_p, \ \lambda \mapstoo \int_{\X} \chi (c\tau^t \cdot x ) d\lambda.
\]
Since $\varphi_{\chi}$ is $U_F$-equivariant, it induces a map in cohomology 
\[
\varphi_\chi\colon H^{n-1}(\Gamma, \D(\X, \Z[1/m])) \too H^{n-1}(U_F, \bar{\Q}_p).
\]
Fix the generator $c_{U_F} \in H_{n-1}(U_F, \Z) \simeq H_{n-1}(X(F), \Z) \simeq \Z$ corresponding to the positive orientation of $X(F)$ in \eqref{eq: X(F)}. We can then consider the cap product
\[
c_{U_F} \frown \varphi_\chi(\mu) \in \bar{\Q}_p.
\]
To make the notation more transparent, we will write
\[
c_{U_F} \frown \varphi_\chi(\mu) = \int_{\X} \chi (c\tau^t \cdot x) d\mu(c_{U_F}).
\]
When $\chi$ is locally constant, this quantity relates to special values of partial $L$-functions in the following way.
\begin{proposition}\label{prop: integral of loc constant function along mu = L-function}
	Let $\chi\colon \mathfrak a_p^{-1}   - p \mathfrak a_p^{-1}\twoheadrightarrow{} \left( \mathfrak a^{-1} - p \mathfrak a^{-1} \right)/ p^r\mathfrak a^{-1} \to \bar{\Q}$ be an $\cO_F^\times$-invariant function. Then, 
	\[
	\int_{\X} \chi(c\tau^t \cdot x)  d\mu(c_{U_F})  = \Delta_c(1_{[\mathfrak a], p} {\chi}, 0).
	\]
\end{proposition}
\begin{proof}
	Consider the $U_F$-equivariant morphism
	\[
	\varphi_{\chi,r}\colon \D(\X_r, \Z[1/m] ) \too \bar{\Q}_p, \ \lambda_r \mapstoo \sum_{\bar{x} \in \X_r} \chi(c \tau^t \cdot x) \lambda_r(\bar{x}).
	\]	
	Since $\chi$ factors through $\left( \mathfrak a^{-1}  - p \mathfrak a^{-1}\right) / p^r\mathfrak a^{-1}$, it follows that
	\begin{equation}\label{eq: varphi_chi(mu) factors through mur}
	c_{U_F} \frown \varphi_{\chi}(\mu) = c_{U_F} \frown \varphi_{\chi, r}(\mu_r),
	\end{equation}
	where $\mu_r \in H^{n-1}(\Gamma, \D(\X_r, \Z[1/m]))$ is the class described in Definition \ref{def: mu_r}. In particular, $c_{U_F} \frown \varphi_{\chi, r}(\mu_r) \in \bar{\Q}$. Fix an embedding $\bar{\Q} \subset \C$. Then, the right-hand side of \eqref{eq: varphi_chi(mu) factors through mur} can be computed using a representative of the image of $\mu_r$ in $H^{n-1}(\Gamma, \D(\X_r, \R))$. By Proposition \ref{prop: representative of mu_r with R coefficients}, such a representative is given by 
	\[
	\varphi_r\colon \Gamma^n \too \D(\X_r, \R), \ (\gamma_0, \dots, \gamma_{n-1}) \mapstoo \left( \bar{x} \mapstoo  \int_{\Delta(\gamma_0 z, \dots, \gamma_{n-1} z)}  (x/p^r)^\ast {}_c E_{\psi} \right),
	\]
	where $z \in \sX$ denotes an arbitrary point and $\Delta(\gamma_0z , \dots , \gamma_{n-1}z)$ is the geodesic simplex in $\sX$ with vertices $\{ \gamma_i z \}_{i}$. Hence, \eqref{eq: varphi_chi(mu) factors through mur} can be written as
	\[
	\int_{X(F)}\iota_\tau^\ast \sum_{\bar{x} \in \X_r} \chi(c \tau^t \cdot x)(x/p^r)^\ast {}_c E_{\psi}  = \int_{X(F)} \iota_\tau^\ast E_{\tau, \chi},
	\]
	where $X(F)$ is given in \eqref{eq: X(F)} and $E_{\tau, \chi}$ in \eqref{eq: definition of E_r, chi}. By Theorem \ref{thm: period of E_r,chi in terms of L-functions}, the result follows.
\end{proof}

Let $\bar{U}_F$ denote the completion of $U_F$ in $\cO_{F, p}^\times$. The previous proposition has an interpretation in terms of measures on $\mathfrak a_p^{-1} - p\mathfrak a_p^{-1} / \bar{U}_F$, that we proceed to explain. Cap product with $c_{U_F}$ yields the morphism
\[
H^{n-1}(U_F , \D(\X, \Z[1/m]) \xlongrightarrow{}  \D(\X, \Z[1/m])_{U_F}.
\]
In addition, if we identify $\X \xrightarrow{\sim} \mathfrak a_p^{-1}$ via dot product with $\tau \in F^n$, and denote $\pi\colon \mathfrak a_p^{-1} - p\mathfrak a_p^{-1} \twoheadrightarrow (\mathfrak a_p^{-1} - p\mathfrak a_p^{-1}) / \bar{U}_F$ the quotient map, we can define
\[
h \colon \D(\X, \Z[1/m])_{U_F} \too \D\left( (\mathfrak a_{p}^{-1} - p \mathfrak a_p^{-1}) /\bar{U}_F , \Z[1/m]\right)
\]
in the following way: for $[\mu] \in \D(\X, \Z[1/m])_{U_F}$ and $V \subset \mathfrak (\mathfrak a_p^{-1} - p\mathfrak a_p^{-1})/\bar{U}_F$ open set, $h([\lambda])(V) \coloneq \lambda(\pi^{-1}(V))$. We can then consider these two maps and view $c_{U_F} \frown \mu_{\vert U_F}$ as an element in $\D(\mathfrak a_p^{-1} - p\mathfrak a_p^{-1} /\bar{U}_F, \Z[1/m])$. 

Moreover, the bijection given in \eqref{eq: from a^-1/p^r a^{-1} to preimage of a in Gp^r} considered for every $r \geq 1$, induces a homeomorphism 
\[
\mathfrak a_p^{-1} - p\mathfrak a_p^{-1} /\bar{U}_F \xlongrightarrow{\sim} H_{\mathfrak a} \subset G = \varprojlim G_{p^r}.
\]
\begin{corollary}\label{cor: relation mu with p-adic L-function}
    Let $\mu_{\mathfrak a}$ be the measure on $H_{\mathfrak a}$ considered in Theorem \ref{thm: Deligne--Ribet p-adic L-function}. Via the homomorphism above, we have
    $c_{U_F} \frown \mu_{\vert U_F} = \mu_{\mathfrak a}$. In particular, for $s \in \Z_p$
    \[
    L_p(1_{[a], p}, s) = \int_{\X} \left \langle \mathrm{N}(\mathfrak a) \Norm_{F_p/\Q_p}(c\tau^t \cdot x)\right \rangle^{-s} d\mu(c_{U_F}).
    \]
\end{corollary}
\begin{proof}
    The equality $c_{U_F} \frown \mu_{\vert U_F} = \mu_{\mathfrak a}$ follows from the discussion above, Proposition \ref{prop: integral of loc constant function along mu = L-function} and Theorem \ref{thm: Deligne--Ribet p-adic L-function}.
\end{proof}

As a consequence, we obtain the relation between the local trace of $J_{E, \mathcal L}[\tau]$ and the local trace of the logarithm of a Gross--Stark unit.
\begin{theorem}\label{thm: Tr J_Eis[tau] = Tr log(u)}
	Let $u \in \cO_H[1/p]_-^\times \otimes \Q$ be the Gross--Stark unit introduced in Proposition \ref{prop: determination of u}. We have,
	\[
	\mathrm{Tr}_{F_p/\Q_p}J_{E, \mathcal L}[\tau] = \mathrm{Tr}_{F_p/\Q_p}\log_p(u^{\sigma_{\mathfrak a}}).
	\]
\end{theorem}
\begin{proof}
	By viewing $\D_0(\X, \Z[1/m]) \subset \D(\X, \Z[1/m])$, we can consider $c_{U_F} \frown \varphi_{\chi_s}(\mu_0) \in \bar{\Q}_p$, where $s \in \Z_p$ and 
    \[
    \chi_s \colon \mathfrak a_p^{-1} - p \mathfrak a_p^{-1} \too \bar{\Q}_p^\times, \ \alpha \mapstoo  \langle \Norm(\mathfrak a) \Norm_{F_p/\Q_p}(\alpha) \rangle^{-s}.
    \]
    Moreover, since $\mu_0$ is a lift of $\mu$, Corollary \ref{cor: relation mu with p-adic L-function} implies that for every $s \in \Z_p$,
	\[
	c_{U_F} \frown \varphi_{\chi_s}(\mu_0) = L_p(1_{[\mathfrak a], p}, s).
	\]
	Then, it follows from the definition of $J_{E, \mathcal L} = \mathrm{ST}(\mu_0)$, and the fact that $\mu_0$ takes values on measures of total mass zero, that $\mathrm{Tr}_{F_p/\Q_p} J_{E, \mathcal L}[\tau] = -L_p'(1_{[\mathfrak a], p}, 0)$.
	Hence, the result follows from Theorem \ref{thm: Gross--Stark conj}.
\end{proof}

\section{Conjecture on the values of the log-rigid class}\label{sec: conjecture on values of JEis}

In this section, we make a conjecture on the values of the log-rigid class $J_{E, \mathcal L}$ at certain points $\tau \in X_p$ attached to totally real fields where $p$ is inert. Then, we study the conjecture for the concrete case that $F/\Q$ is Galois and the point $\tau$ corresponds to a $\Gal(F/\Q)$-stable ideal of $F$. Finally, we provide an observation that motivates the conjecture.

\subsection{Conjecture on the values $J_{E, \mathcal L}[\tau]$}
We consider the same notation as in Section \ref{sec: Values of the log-rigid class and Gross--Stark conjecture}. In particular, let $F$ be a totally real field where $p$ is inert, let $\mathfrak a$ be an integral ideal of $F$ coprime to $pc$, and fix $\tau \in F^n$ a vector whose entries give an oriented $\Z$-basis of $\mathfrak a^{-1}$, which yields a point $\tau \in X_p$. We can then consider $J_{E, \mathcal L}[\tau] = c_{U_F} \frown \mathrm{ev}_\tau(J_{E, \mathcal L}) \in F_p$, where $c_{U_F} \in H_{n-1}(U_F, \Z) \simeq H_{n-1}(X(F), \Z)$  is a generator corresponding to the positive orientation of $X(F)$ in \eqref{eq: X(F)}. Moreover, let $H$ be the narrow Hilbert class field of $F$, $\fp$ the fixed prime ideal of $H$ above $p$ determined by the embedding $\bar{\Q} \subset \bar{\Q}_p$, and recall the inclusion $H \subset H_\fp = F_p$. Also note that the $p$-adic logarithm can be extended to a map 
\[
\log_p \colon H_\fp^\times \otimes \Q \too H_\fp
\]
by linearity. In view of Theorem \ref{thm: Tr J_Eis[tau] = Tr log(u)}, we make the following conjecture.

\begin{conjecture}\label{conj: Jtriv[tau] = log(GS)}
	Suppose that $F$ is a totally real field where $p$ is inert. Let $\tau \in X_p$ be as above and let $u \in \cO_H[1/p]_-^\times \otimes \Q$ be the Gross--Stark unit determined in Proposition \ref{prop: determination of u}. Then,
	\[
	J_{E, \mathcal L}[\tau] = \log_p(u^{\sigma_{\mathfrak a}}),
	\]
    where $\sigma_{\mathfrak a} \in \Gal(H/F)$ denotes the Frobenius associated to the class of $\mathfrak a$.
\end{conjecture} 
When $n = 2$, the conjecture is true by Theorem B of \cite{DPV2} since, as stated in Remark \ref{remark: comparison J_Eis with JDR when n = 2}, we have the equality $J_{E, \mathcal L} = \log_p(\mathcal J_{\mathrm{DR}})$.

\begin{remark}
    Let $\alpha \in F$ be such that $\Norm_{F/\Q}(\alpha)$ is positive. Then both $\tau$ and $\alpha \tau$ give rise to oriented basis of an ideal and to the same point in $X_p$. In particular $J_{E, \mathcal L}[\tau] = J_{E, \mathcal L}[\alpha \tau]$. Moreover, the fact that Gross--Stark units have absolute value $1$ for any complex embedding of $H$, and that $\sigma_{\mathfrak a}$ and $\sigma_{\alpha^{-1}\mathfrak a}$ differ by an even number of complex conjugations, imply that $u^{\sigma_\mathfrak a} = u^{\sigma_{\alpha^{-1}\mathfrak a}}$. This observation suggests that $X_p$ is suitable to encode Gross--Stark units.
\end{remark}

\subsection{The case of Galois extensions}\label{subsec: The case of Galois extensions}

Suppose that $F$ is Galois over $\Q$. If the narrow ideal class $[\mathfrak a]$ is $\Gal(F/\Q)$-stable, we prove that $\log_p(\sigma_{\mathfrak a} u) \in \Q_p$. If moreover the ideal $\mathfrak a$ is $\Gal(F/\Q)$-stable, we show that $J_{E, \mathcal L}[\tau] \in \Q_p$. Thus, Conjecture \ref{conj: Jtriv[tau] = log(GS)} follows from Theorem \ref{thm: Tr J_Eis[tau] = Tr log(u)} in the case that $\mathfrak a$ is $\Gal(F/\Q)$-stable. 

Observe that under these assumptions, $H$ is Galois over $\Q$. Denote by $D_{\fp} \subset \Gal(H/\Q)$ the decomposition group at $\fp$. Note that $\Gal(H/\Q)$, and therefore also $D_\fp$, act on $\cO_H[1/p]^\times_{-} \otimes \Q$.

\begin{lemma}\label{lemma: eta(sigma_a u) = eta(u) in Up(x)Q}
    Let $u$ be the Gross--Stark unit as above and let $[\mathfrak a]$ be a narrow ideal class that is $\Gal(F/\Q)$-fixed. For every ${\eta} \in D_\fp$, we have $\eta(\sigma_\mathfrak a u) = \sigma_\mathfrak a u$ in $\cO_H[1/p]^\times_{-} \otimes \Q$.
\end{lemma}
\begin{proof}
    We will use the uniqueness property determining Gross--Stark units of Proposition \ref{prop: determination of u}. For every ideal $\mathfrak b$ of $\cO_F$, denote by $\sigma_{\mathfrak b} \in \Gal(H/F)$ the corresponding Frobenius and observe
    \[
    \sigma_{\mathfrak b} \eta \sigma_{\mathfrak a} (u) = \eta \eta^{-1} \sigma_{\mathfrak b} \eta  \sigma{_\mathfrak a} (u) = \eta \sigma_{\eta^{-1}(\mathfrak b)} \sigma_{\mathfrak a}(u) = \eta \sigma_{\eta^{-1}(\mathfrak b\mathfrak a)}(u), 
    \]
    where we used the $\Gal(F/\Q)$-equivariance of the Artin map in the second equality, and the fact that $[\mathfrak a]$ is Galois fixed in the last one. From there,
	\[
	\begin{split}
		\ord_{\mathfrak p}(\sigma_{\mathfrak b}{\eta}\sigma_\mathfrak a (u) ) & = \ord_{\mathfrak p}(\eta \sigma_{\eta^{-1}(\mathfrak b\mathfrak a)}(u)) \\ & = \ord_{\mathfrak p}( \sigma_{\eta^{-1}(\mathfrak b\mathfrak a)}(u)) \\ & = \Delta_{c}(1_{[\eta^{-1}(\mathfrak b\mathfrak a)]} , 0) \\& = \Delta_c(1_{[\mathfrak b \mathfrak a]}, 0),
	\end{split}
	\]
	where we used that $\tilde{\eta}(\mathfrak p) = \mathfrak p$ in the second equality, Proposition \ref{prop: determination of u} in the second to last equality, and the last equality follows from $L(1_{[\eta^{-1}(\mathfrak b\mathfrak a)]}, s) = L(1_{[\mathfrak b \mathfrak a]}, s)$, which can be verified from their definition. From the uniqueness asserted in Proposition \ref{prop: determination of u}, it can be deduced $\eta(\sigma_{\mathfrak a }u) = \sigma_\mathfrak a u$ in $\cO_H[1/p]^\times_- \otimes \Q$ and we are done.      
\end{proof}

\begin{proposition}\label{lemma: log(sigma_a u) is in Qp if [a] is Gal(F/Q)-stable}
    Let $u \in \cO_H[1/p]_{-}^\times \otimes \Q$ be the Gross--Stark unit introduced above and let $[\mathfrak a]$ be a narrow ideal class that is $\Gal(F/\Q)$-fixed. We have $\log_p(\sigma_{\mathfrak a} u) \in \Q_p$.
\end{proposition}
\begin{proof}
    We will see that for every $\tilde{\eta} \in \Gal(H_\fp/\Q_p)$, we have $\tilde{\eta} \log_p(\sigma_{\mathfrak a} u) = \log_p(\sigma_{\mathfrak a} u)$.
    Consider the isomorphism given by extending an automorphism in $D_\fp \subset \Gal(H/\Q)$ to $H_{\fp}$.
    \[
    D_\fp \xlongrightarrow{\sim} \Gal(H_\fp / \Q_p), \ \eta \mapstoo \tilde{\eta}.
    \]
    Observe that the map induced by the $p$-adic logarithm
    \[
    \log_p \colon \cO_H[1/p]^\times_- \otimes \Q \too H_{\fp}^\times \otimes \Q \too H_\fp
    \]
    satisfies the following invariance property: for every $\eta \in D_\fp$ and $x \in \cO_H[1/p]^\times_- \otimes \Q$, we have $\log_p(\eta x) = \tilde{\eta} \log_p(x)$. Indeed, this follows from the definition of $D_\fp$ and the $\Gal(H_\fp/\Q_p)$-invariance of the $p$-adic logarithm on $H_\fp^\times$. Applying this to $x = \sigma_{\mathfrak a}(u)$ and using Lemma \ref{lemma: eta(sigma_a u) = eta(u) in Up(x)Q} we obtain the desired result.
\end{proof}
\begin{remark}\label{rmk: from log(u) to u and narrow Hilbert class field Galois case}
Suppose that $c$ is as in Remark \ref{rmk: from Gross--Stark to Brumer--Stark}. Then, $u \in \cO_H[1/p]^\times_-$ and its image under the embedding $u \in H^\times \subset H_{\fp}^\times = F_p^\times$ lands in $\Q_p^\times$. Therefore, by Gross--Stark,
\[
\log_p(u) = -\frac{1}{n} L_p'(1_{[\cO_F], p} , 0).
\]
In particular, $\frac{1}{n} L_p'(1_{[\cO_F], p} , 0) \in p\Z_p$, so it is in the domain of the $p$-adic exponential $\exp_p$. Moreover, since the valuation of $u$ at $\fp$ is equal to $\Delta_c(1_{[\cO_F]},0 )$, we have that, up to a root of unity in $F_p^\times$ (see \cite[Remark 2.7]{DKExplicitCFT} for a discussion on this ambiguity), 
\[
u = p^{\Delta_c(1_{[\cO_F]} ,0 )} \exp_p\left(  -\frac{1}{n} L_p'(1_{[\cO_F], p} , 0)\right),
\]
This gives an explicit formula for $u$, generalizing the type of abelian extensions of $F$ that can be constructed only from $p$-adic $L$-functions, and in particular generalizing Proposition 3.14 of \cite{Gross1981padicL}, see also Remark 7 of \cite{DDP} (note that the field generated by $u$ over $F$ is the same as the field generated by all its $\Gal(H/F)$-conjugates, as the extension is abelian). This is studied in more detail in \cite{RX2.5}.
\end{remark}

We now study the invariant $J_{E, \mathcal L}[\tau]$ in the case that the ideal $\mathfrak a$ is $\Gal(F/\Q)$-stable. Since the coordinates of $\tau \in F^n$ give a $\Z$-basis of $\mathfrak a^{-1}$, the isomorphism \eqref{eq: iso beta Q^n = F} induces an embedding
\[
\cO_F^\times \rtimes \Gal(F/\Q) \intoo \GL_n(\Z)
\]
determined by the following equations: for every $x \in \Q^n$, $\alpha \in F^\times$ and $\eta \in \Gal(F/\Q)$,
\[
\alpha (\tau^t \cdot x) = \tau^t A_{\alpha}x , \ \eta(\tau^t \cdot x) = \tau^t A_{\eta} x.
\]
Denote $\D_0 \coloneq \D_0(\X, \Z[1/m])$. Recall that $\GL_n(\Z)$ acts on $\D_0$ as follows: for $g \in \GL_n(\Z)$, $\lambda \in \D_0$, and $U \subset \X$ compact open 
\[
(g \cdot \lambda)(U) \coloneq \lambda(g^{-1} U).
\]
Consider also the $\GL_n(\Z)$-module $\D_0(\det) \coloneq \D_0 \otimes_{\Z[1/m]} \Z[1/m](\det)$. We use these actions and the embedding above to define an action of $\cO_F^\times$ and of $\Gal(F/\Q)$, on $\D_0$ and $\D_0(\det)$. In particular, since $\{1\} \rtimes \Gal(F/\Q)$ normalizes $U_F \rtimes \{1\}$, we have natural actions of $\Gal(F/\Q)$ on $H^{n-1}(U_F, \D_0(\det))$ as well as on the coinvariants $(\D_0)_{U_F}$. 

Recall the class $\mu_0 \in H^{n-1}(\Gamma , \D_0)_{\Q}^-$ given in \eqref{eq: mu0 fixed}. To lighten the notation for the next proof, by avoiding the appearance of tensor products, let $\ell \in \Z_{\geq 0}$ be such that $\ell \mu_0$ lifts to an element in $H^{n-1}(\Gamma, \D_0)^-$. Fix such a lift and denote it by $\tilde{\mu}_0 \in H^{n-1}(\Gamma, \D_0)^-$. Note that 
\[
\ell J_{E, \mathcal L}[\tau] = \int_{\X} \log_p(\tau^t \cdot x ) d\lambda \in F_p, 
\]
where $\lambda \coloneq c_{U_F} \frown \tilde{\mu}_0 \in (\D_0)_{U_F}$ and $c_{U_F} \in H_{n-1}(U_F, \Z)$ is the fixed fundamental class. This is independent of the choice of lift $\tilde{\mu}_0$, as the difference between two such lifts is torsion.

\begin{lemma}\label{lemma: c_{U_F} cap mu0 is GF-invariant}
    Let $\tilde{\mu}_0 \in H^{n-1}(\Gamma, \D_0)$ and $c_{U_F} \in H_{n-1}(\Gamma, \Z)$ be as above. The element $\lambda = c_{U_F} \frown \tilde{\mu}_0 \in (\D_0)_{U_F}$ is fixed by $\Gal(F/\Q)$.
\end{lemma}
\begin{proof}
    By Shapiro's lemma, $\tilde{\mu}_0 \in H^{n-1}(\Gamma, \D_0)^{-}$ admits a unique lift via the isomorphism given by restriction
    \[
    H^{n-1}(\GL_n(\Z), \D_0(\det)) \xlongrightarrow{\sim} H^{n-1}(\Gamma, \D_0)^-,
    \]
    that we will also denote by $\tilde{\mu}_0$. It then follows that the restriction of $\tilde{\mu}_0$ to $U_F$ is $\Gal(F/\Q)$-invariant, as it can be obtained as the image of $\tilde{\mu}_0$ via the following maps
    \[
	H^{n-1}(\GL_n(\Z), \D_0(\det)) \too H^{n-1}(U_F \rtimes \Gal(F/\Q), \D_0(\det)) \too H^{n-1}(U_F, \D_0(\det))^{\Gal(F/\Q)}.
	\]
    The result follows as cup product with $c_{U_F}$ induces a $\Gal(F/\Q)$-equivariant map
    \[
    H^{n-1}(U_F, \D_0(\det)) \xlongrightarrow{\sim} (\D_0)_{U_F},
    \]
    which can be verified via a direct calculation.
\end{proof}

\begin{theorem}\label{thm: J_Eis[tau] in Qp if tau is Gal stable}
	Suppose that the coordinates of $\tau \in F^n$ given an oriented $\Z$-basis of a $\Gal(F/\Q)$-stable ideal $\mathfrak a^{-1}$. We have, $J_{E, \mathcal L}[\tau] \in \Q_p$.
\end{theorem}
\begin{proof}
	We need to see that $J_{E, \mathcal L}[\tau] \in F_p$ is fixed by $\Gal(F_p/\Q_p)$. For every $\tilde{\eta} \in \Gal(F_p/\Q_p)$ denote by $\eta \in \Gal(F/\Q)$ its restriction to $F$ and note
	\[
	\begin{split}
		\ell J_{E, \mathcal L}[\tau]^{\tilde{\eta}} & = \int_\X \log_p(\eta( \tau^t \cdot x) ) d\lambda \\ & = \int_{\X} \log_p( \tau^t A_{\eta} x ) d\lambda \\ & = \int_{\X} \log_p( \tau^t \cdot x )d(A_{\eta} \cdot \lambda) = \ell J_{E, \mathcal L}[\tau], 
	\end{split} 
	\]
	where in the last equality we used that $\lambda \in (\D_0)_{U_\tau}$ is fixed by $\Gal(F/\Q)$ by Lemma \ref{lemma: c_{U_F} cap mu0 is GF-invariant}.
\end{proof}
\begin{remark}
    In Theorem \ref{thm: J_Eis[tau] in Qp if tau is Gal stable}, we only used that $\tilde{\mu}_0$ is a cohomology class for $\SL_n(\Z)$ that belongs to the $w = -1$ eigenspace. Thus, it can be applied to other rigid analytic classes. 
\end{remark}
\begin{corollary}
    Suppose that $F$ is a totally real field that is Galois over $\Q$ and where $p$ is inert. Let $\tau \in F^n$ with coordinates generating $\mathfrak a^{-1}$, where $\mathfrak a$ is a $\Gal(F/\Q)$-stable ideal, and let $u \in \cO[1/p]_-^\times \otimes \Q$ be the Gross--Stark unit of Proposition \ref{prop: determination of u}. We have,
    \[
    J_{E, \mathcal L}[\tau] = \log_p(u^{\sigma_\mathfrak a}).
    \]
\end{corollary}

\subsection{Example in the Galois case} In this section, we give a numerical computation that exemplifies Proposition \ref{lemma: log(sigma_a u) is in Qp if [a] is Gal(F/Q)-stable}, namely the fact that the conjugates of Gross--Stark units corresponding to $\Gal(F/\Q)$-fixed narrow ideal classes belong to $\Q_p$. We used the algorithm developed by Damm--Johnsen, see \cite{Damm--Johnsen}, which is publicly available. We made minor modifications to the algorithm to output all conjugates of a given Gross--Stark unit. This algorithm computes non $c$-smoothed units, but we will abuse notation and denote the Gross--Stark units with the same symbol as before, as they only differ by a power.

Concretely, we take $p = 3$ and $F = \Q(\sqrt{D})$ with $D = 689$. The narrow Hilbert class field of $F$, denoted by $H$, is a cyclic extension of $F$ of degree $8$. The Galois group $\Gal(H/F)$ is isomorphic to the narrow Hilbert class group of $F$, denoted by $G_1$. Denote by $\mathcal F_D$ the set of binary quadratic forms with integer coefficients and discriminant $D$. This set is equipped with a group action of $\SL_2(\Z)$ by linear transformations and we have a bijection 
\[
\mathcal F_D/ \SL_2(\Z) \xlongrightarrow{\sim} G_1
\]
\[
[Q] = [a, b, c] \mapstoo [\mathfrak a_Q] = \left[ \left( a, \frac{-b + \sqrt{D}}{2} \right) \right],
\]
where $[a,b,c]$ denotes the class of $ax^2 + bxy + cy^2$. The Gross--Stark units are, up to high $p$-adic precision, roots of the polynomial
\[
6561x^8 - 11340x^7 - 882x^6 + 4333x^5 + 2665x^4 + 4333x^3 - 882x^2 - 11340x + 6561.
\]
More precisely, for every class $[Q] \in \mathcal F_D/\SL_2(\Z)$, the table below gives the image of the Gross--Stark unit $\sigma_{\mathfrak{a}_{Q}}(u) \in H$ via the embedding $H \subset H_{\fp} = F_p$.

\begin{table}[H]
 \caption{$p = 3$, $D = 689$. Elements in $\mathcal F_D/\SL_2(\Z)$ and their Gross--Stark unit.}
	\begin{tabular}{|l|l|l|}
		\hline
		$[Q]$  & $\mathrm{ord}(\mathfrak [\mathfrak{a}_Q])$ & $\sigma_{\mathfrak{a}_Q}(u) \in F_p$                                                    \\ \hline
		$[-20, 17, 5]$  & $8$ & $3^{-2}(7283498230698546457 + 20427811426324513506\sqrt{D}) + O(3^{39})$ \\
		$[-10, 7, 16]$  & $2$ & $3^4 \cdot 28799930840163216397 + O(3^{45})$                             \\
		$[-10, 17, 10]$ & $4$ & $25613292858296352193 + 34405602800800679412\sqrt{D}+ O(3^{41})$         \\
		$[-5, 17, 20]$  & $8$ & $3^2 (28389335835840796072 + 1041259434467889369\sqrt{D}) + O(3^{43})$   \\
		$[5,17, -20]$   & $8$ & $3^{-2} (7283498230698546457 + 16045184950846272897\sqrt{D} + O(3^{39})$ \\
		$[10,7, -16]$   &  $1$ &  $3^{-4}\cdot 23094469614450736543 + O(3^{37})$                           \\
		$[10, 17, -10]$  & $4$ & $25613292858296352193 + 2067393576370106991\sqrt{D} + O(3^{41})$         \\
		$[20, 17, -5]$  & $8$ & $3^2 (28389335835840796072 + 35431736942702897034\sqrt{D}) + O(3^{43})$   \\ \hline
	\end{tabular}
\end{table}
Note that $\sigma_{\mathfrak a_Q}(u) \in \Q_p$ if and only if $[\mathfrak a_Q]$ is $2$-torsion in $G_1$. For real quadratic fields, this is equivalent to the fact that the class $[\mathfrak a_Q]$ is $\Gal(F/\Q)$-fixed, as predicted by Proposition \ref{lemma: log(sigma_a u) is in Qp if [a] is Gal(F/Q)-stable}.

The work of Damm--Johnsen made it possible to verify this phenomenon in many additional cases. In these cases, the size of the narrow Hilbert class group ranged from $2$ to $20$.

\subsection{Further comments}
We conclude with some observations to support the conjecture. Denote $\D \coloneq \D(\X, \Z[1/m])$ and $\D_0 \coloneq \D_0(\X, \Z[1/m])$. Recall the class $\mu \in H^{n-1}(\Gamma, \D)^{-}$ constructed in Section \ref{sec: Topological construction of the Eisenstein group cohomology class} and denote by $\mu_{\vert U_F}$ its restriction to $U_F \subset \Gamma$. Consider the $U_F$-equivariant morphism 
\[
\bar{\mathcal E}\colon \D \too F_p /\Z[1/m] \log_p(\cO_F^\times), \  \lambda \mapstoo \int_{\X} \log_p(c\tau^t \cdot x) d\lambda,  
\]
where $\Z[1/m] \log_p(\cO_F^\times)$ denotes the $\Z[1/m]$-span of $\log_p(\cO_F^\times)$. Theorem \ref{thm: Tr J_Eis[tau] = Tr log(u)} implies 
\begin{equation}\label{eq: formula u_tau mod Z_plog(O_F)}
c_{U_F} \frown \bar{\mathcal E}(\mu_{\vert U_F}) = \log_p(u^{\sigma_\mathfrak a}) \mod \Z_p \log_p(\cO_F^\times).
\end{equation}

Conjecture \ref{conj: Jtriv[tau] = log(GS)} predicts an expression for  $\log_p(u^{\sigma_{\mathfrak a}})$ without the ambiguity $\Z_p \log_p(\cO_F^\times)$. Observe that, if we consider measures of total mass zero, we can define 
\[
\mathcal{E}\colon \D_0 \too F_p, \ \lambda \mapstoo \int_\X \log_p(c\tau^t \cdot x) d \lambda,
\]
which is $U_F$-equivariant. Moreover, it follows from Proposition \ref{prop: image of mu is torsion} that $\mu_{\vert U_F}$ lifts to a class in $H^{n-1}(U_F, \D_0)$. However, the lift is not unique. Indeed, the long exact sequence
\[
\cdots \too H^{n-2}(U_F, \Z[1/m]) \xlongrightarrow{\delta} H^{n-1}(U_F, \D_0) \too H^{n-1}(U_F, \D) \too \cdots
\]
shows that a lift of $\mu_{\vert U_F}$ is well-defined up to the image of $\delta$. Since $U_F \simeq \Z^{n-1}$ by Dirichlet's unit theorem, we have a natural isomorphism
\begin{equation}\label{eq: description H^{n-2}(U_F)}
H^{n-2}(U_F, \Z[1/m]) \simeq H_1(U_F, \Z[1/m]) \simeq U_F \otimes \Z[1/m].
\end{equation}
This leads to the following proposition. 
\begin{proposition}
	The map 
	\[
	H^{n-2}(U_F, \Z[1/m]) \xlongrightarrow{} F_p, \ \varepsilon \mapstoo  c_{U_F} \frown \mathcal E( \delta(\varepsilon) )
	\]
    has image equal to $\Z[1/m]\log_p(U_F)$. More precisely, via the natural isomorphism given in \eqref{eq: description H^{n-2}(U_F)}, it is equal to $\log_p \colon U_F \otimes \Z[1/m] \xrightarrow{\sim} \Z[1/m]\log_p(U_F)$.
\end{proposition}

In other words, the process of lifting $\mu_{\vert U_F}$ to a class valued in total mass zero measures allows to compute its image under $\mathcal E$ and construct an element in $F_p$. However, since the lift is only well-defined up to $U_F \otimes \Z[1/m]$, the elements we construct in $F_p$ are only well-defined up to $\Z[1/m] \log_p(U_F)$. Thus, we obtain a similar ambiguity for the Gross--Stark unit as the one appearing on the formula of the Gross--Stark conjecture.

However, in this paper, we considered cohomology classes for $\Gamma$, instead of $U_F$, to define our invariants. In this way, we obtained that $\mu \in H^{n-1}(\Gamma, \D)_{\Q}^{-}$ has a unique lift $\mu_0 \in H^{n-1}(\Gamma, \D_0)_{\Q}^{-}$. Indeed, as explained in Section \ref{subsec: lfting to measures of total mass zero}, this follows from the long exact sequence
\[
H^{n-2}(\Gamma, \Z[1/m])^{-} \too H^{n-1}(\Gamma, \D_0)^{-} \too H^{n-1}(\Gamma , \D)^{-} \too H^{n-1}(\Gamma, \Z[1/m])^{-},
\] 
Proposition \ref{prop: image of mu is torsion}, and the fact that $H^{n-2}(\Gamma, \Z_p)^{-}$ is torsion by \cite{LS}. Then, the restriction ${\mu_0}_{\vert U_F}$ is a preferred lift of $\mu_{\vert U_F}$ to $H^{n-1}(U_F, \D_0(\X, \Z[1/m]))$. Hence, using ${\mu_{0}}_{\vert U_F}$ and the map $\mathcal E$, we are able to produce a canonical element in $F_p$
\[
c_{U_F} \frown \mathcal E( {\mu_0}_{\vert U_F}  ) = J_{E, \mathcal L}[\tau] \in F_p.
\]
The fact that this construction is unique suggests that the quantity we produced could be a preferred lift of $\mathrm{Tr}_{F_p/\Q_p} \log_p(u^{\sigma_{\mathfrak a}})$, and this motivates us to state Conjecture \ref{conj: Jtriv[tau] = log(GS)} above. 

We summarize this discussion with the following commutative diagram
\[
\begin{tikzcd}
	\text{torsion} \arrow[r] \arrow[d]                                                                        & {H^{n-1}(\Gamma, \D_0)^{-}} \arrow[r] \arrow[d]                      & {H^{n-1}(\Gamma, \D)^{-}} \arrow[d]                            \\
	U_F \otimes \Z[1/m] \arrow[d, "c_{U_F} \frown \mathcal E \circ \delta(\cdot)"] \arrow[r, "\delta"] & {H^{n-1}(U_F, \D_0)} \arrow[d, "c_{U_F} \frown \mathcal E(\cdot)"] \arrow[r] & {H^{n-1}(U_F, \D)} \arrow[d, "c_{U_F} \frown \bar{\mathcal E}(\cdot)"] \\
	\Z[1/m]\log_p(U_F)  \arrow[r]                                                            & F_p \arrow[r]                                                                     & F_p/(\Z[1/m]\log(U_F)).                                                
\end{tikzcd}
\] 

{
\printbibliography
}
\end{document}